\documentclass[10pt]{amsart}
\usepackage[a4paper,margin=22mm]{geometry}
\usepackage{amssymb}
\usepackage{booktabs}
\usepackage{array}
\usepackage{xcolor}
\usepackage[protrusion=true,expansion=false]{microtype}
\usepackage[hidelinks,bookmarks=false]{hyperref}
\newtheorem{theorem}{Theorem}[section]
\newtheorem{proposition}[theorem]{Proposition}
\newtheorem{lemma}[theorem]{Lemma}
\theoremstyle{definition}
\newtheorem{definition}[theorem]{Definition}
\newtheorem{certificate}[theorem]{Proposition}
\theoremstyle{remark}

\newcommand{\dd}{\,\mathrm d}
\newcommand{\cC}{\mathcal C}

\title[Beyond Liu's threshold]{Beyond Liu's $0.382709$ threshold\\
for the union-closed sets conjecture}
\author{Simone Costa}
\address{DICATAM, Universit\`a degli Studi di Brescia, Via Branze 43, 25123 Brescia, Italy}
\email{simone.costa@unibs.it}
\author{Ankan Sadhu}
\address{Government College of Engineering and Ceramic Technology, Kolkata, India}
\email{ankan-sadhu-br24-3003@gcect.ac.in}
\date{}
\hypersetup{
 pdftitle={Beyond Liu's 0.382709 threshold for the union-closed sets conjecture},
 pdfauthor={Simone Costa and Ankan Sadhu},
 pdfsubject={Union-closed sets conjecture, entropy inequalities and rigorous interval certificates},
 pdfkeywords={union-closed sets, entropy method, conditionally IID coupling, Liu threshold, interval arithmetic, MPFR}
}

\begin{document}
\begin{abstract}

The union-closed sets conjecture asks whether every finite union-closed family containing a nonempty set has an element contained in at least half of its members.  Let $c_{\mathrm{UC}}$ denote the largest universal lower bound on this proportion.  The best previously proved lower bound is due to Liu and is above $0.3823455$; his own numerical optimization suggested a stronger threshold $c_{\mathrm L}>0.382709087918735$.

The main argument of this paper is to use different protocols (suitable families of probability measures) within Liu's framework.  In particular, we combine the independent protocol with Liu's Example~5 to prove the conjectured bound
\[
 c_{\mathrm{UC}}\ge c_{\mathrm L}.
\]
We then vary the parameter in Example~5 and optimize the choice of the two coefficients in the combination to obtain
\[
 c_{\mathrm{UC}}\ge0.3828852549667978.
\]
We also show that the inequality used in this proof cannot give a threshold larger than $0.3828852599667\ldots$.

\end{abstract}

\maketitle

\section{Introduction}

A finite family $\mathcal F$ of sets is \emph{union-closed} if $A\cup B\in\mathcal F$ whenever $A,B\in\mathcal F$.  The union-closed sets conjecture, attributed to Frankl, asserts that every finite union-closed family containing a nonempty set has an element that belongs to at least half of its members; see the survey of Bruhn and Schaudt~\cite{bruhn-schaudt}.  For such a family, define
\[
 \operatorname{freq}(\mathcal F):=
 \max_{i\in\bigcup\mathcal F}
 \frac{|\{A\in\mathcal F:i\in A\}|}{|\mathcal F|}
\]
and
\[
 c_{\mathrm{UC}}:=\inf_{\mathcal F}\operatorname{freq}(\mathcal F),
\]
where the infimum ranges over all finite union-closed families containing a nonempty set.  The family $\{\varnothing,\{1\}\}$ gives $c_{\mathrm{UC}}\le1/2$, so the conjecture is equivalent to $c_{\mathrm{UC}}=1/2$.

Recent progress on $c_{\mathrm{UC}}$ is based on entropy inequalities.  Gilmer~\cite{gilmer} proved the first absolute constant.  The sharp one-coordinate i.i.d.\ inequality proved independently by Chase--Lovett~\cite{chase-lovett}, Alweiss--Huang--Sellke~\cite{ahs}, and Sawin~\cite{sawin} yields
\[
 c_{\mathrm{UC}}\ge\frac{3-\sqrt5}{2}>0.381966.
\]
Sawin showed that dependence gives a strict improvement, and Yu~\cite{yu} and Cambie~\cite{cambie} made this quantitative, obtaining a value about $0.3823455$.  Liu~\cite{liu} introduced a broader class of conditionally i.i.d.\ couplings and analytically improved this value.  His subsequent finite-dimensional numerical optimization suggested a stronger threshold near $0.382709$.

The main argument of this paper is to use different protocols together within Liu's framework.  This general strategy of combining complementary estimates is also used by Della Fiore--Costa--Dalai~\cite[Section~5]{della-fiore-costa-dalai} in the study of $(b,k)$-hashing.  We combine the independent protocol with Liu's Example~5 to prove Liu's conjectured threshold.  We then vary the parameter in Example~5 and choose the two coefficients in the combination to obtain the stronger bound below.

Throughout, $\log$ denotes the natural logarithm and $H$ denotes the binary entropy function
\[
 H(x):=-x\log x-(1-x)\log(1-x),\qquad 0\le x\le1,
\]
{with the convention that $0\log0=0$.}

Let
\[
 p(x):=x^4-2x^3+3x^2-1.
\]
Since
\[
 p'(x)=2x(2x^2-3x+3)>0\qquad(0<x\le1),
\]
the polynomial $p$ has a unique zero $x_*\in(0,1)$.  Define
\[
 c_{\mathrm L}:=
 1-\frac{x_*H(x_*)}{H(x_*^2)}.
\]
This is the value corresponding to Liu's optimizer.  {Our first theorem proves this conjectured bound.}

\begin{theorem}
\label{thm:liu-bound}
\[
 c_{\mathrm{UC}}\ge c_{\mathrm L}>0.382709087918735.
\]
\end{theorem}

Allowing the parameter in Example~5 to vary in the combined construction gives our main quantitative improvement.

\begin{theorem}
\label{thm:main}
\[
 c_{\mathrm{UC}}\ge0.3828852549667978.
\]
\end{theorem}

{Section~7 shows that the inequality used in the proof of Theorem~\ref{thm:liu-bound} cannot give a threshold larger than
\[
 \widehat c=0.3828852599667\ldots.
\]
Thus the bound in Theorem~\ref{thm:main} is within $5.1\cdot10^{-9}$ of this upper limit.}

\medskip
\noindent\textbf{Roadmap.}
Section~2 states Liu's criterion, defines the probability measures associated with his conditionally i.i.d.\ protocols, and introduces the parameterized form of Example~5.  Section~3 proves Theorem~\ref{thm:liu-bound} from two inequalities in two real variables, established in Sections~4 and~5.  {Section~6 proves Theorem~\ref{thm:main} by varying the parameter in Example~5 within the combined construction.  Section~7 studies limitations of the product estimates used in these arguments.}  Appendix~A states every computer-assisted assertion used in the proofs.  The computational supplement, together with the complete ancillary package of verification programs and outputs, is available in the Zenodo record associated with this work~\cite{joint-zenodo}.

\section{Liu's criterion and Liu's Example~5}

The goal of this section is to state Liu's one-coordinate criterion, record the positive-correlation property of conditionally i.i.d.\ induced laws, and introduce the parameterized form of Example~5 used later.

A \emph{coupling} of probability measures $\nu_1$ and $\nu_2$ on $[0,1]$ is a probability measure $\rho$ on $[0,1]^2$ whose first marginal is $\nu_1$ and whose second marginal is $\nu_2$.  Write $\nu_1\otimes\nu_2$ for the product measure and $\nu^{\otimes2}:=\nu\otimes\nu$.  For $s\in[0,1]$, $\operatorname{Ber}(s)$ denotes the Bernoulli distribution on $\{0,1\}$ with probability $s$ at $1$.

\begin{definition}
Following Liu~\cite[Definition~1]{liu}, a \emph{protocol} is a family
\[
 \Pi=(\Pi_{s,t})_{(s,t)\in[0,1]^2}
\]
such that $\Pi_{s,t}$ is a coupling of $\operatorname{Ber}(s)$ and $\operatorname{Ber}(t)$.
\end{definition}

In Liu's sequential construction~\cite[Definition~1]{liu}, a protocol is applied coordinate by coordinate to generate two copies of a prescribed $\{0,1\}^n$-valued random vector.  At a fixed coordinate, the conditional probabilities of the two output bits have the same marginal distribution, say $\mu$, and their joint distribution is the induced law at that coordinate.  Conditional on values $s,t$, the probability that the union bit equals $0$ is $\Pi_{s,t}(0,0)$.

The following is Proposition~3 of Liu~\cite{liu}, written with natural-logarithm entropy.

\begin{proposition}[Liu]
\label{prop:liu-criterion}
Let $c>0$.  For $j=1,\ldots,N$, let $\Pi^{(j)}$ be a protocol and let $w_j\ge0$ with $\sum_jw_j=1$.  For every probability measure $\mu$ on $[0,1]$, let $\mathcal K_j(\mu)$ be a nonempty set of couplings of $\mu$ with itself containing every induced law produced by $\Pi^{(j)}$ with common marginal $\mu$.

Suppose that there is $\gamma>1$ such that, for every probability measure $\mu$ on $[0,1]$ satisfying
\[
 \int s\,\mu(\dd s)\le c,
\]
one has
\[
 \sum_{j=1}^N w_j
 \inf_{\rho\in\mathcal K_j(\mu)}
 \iint H\bigl(\Pi^{(j)}_{s,t}(0,0)\bigr)\,\rho(\dd s,\dd t)
 \ge \gamma\int H(s)\,\mu(\dd s).
\]
Then $c\le c_{\mathrm{UC}}$.
\end{proposition}

A protocol is \emph{conditionally i.i.d.} in the sense of Liu~\cite[Definition~2]{liu} if the two output bits become independent after conditioning on an auxiliary variable.  Liu shows that every induced law of such a protocol with common marginal $P$ belongs to
\begin{equation}
 \cC_3(P):=
 \left\{\rho:
 \begin{array}{l}
 \rho\text{ is a coupling of }P\text{ with itself, and}\\
 \rho\in\overline{\operatorname{conv}}
 \{Q^{\otimes2}:Q\text{ is a Borel probability measure on }[0,1]\}
 \end{array}
 \right\}.
 \label{eq:C3-def}
\end{equation}
Here the closure is taken with respect to weak convergence of probability measures on $[0,1]^2$.

\begin{lemma}
\label{lem:positive-correlation}
Let $P$ be a probability measure on $[0,1]$, let $\rho\in\cC_3(P)$, and let $\psi:[0,1]\to\mathbb R$ be continuous.  Then
\[
 \iint \psi(x)\psi(y)\,\rho(\dd x,\dd y)
 \ge\left(\int\psi\,\dd P\right)^2.
\]
\end{lemma}

\begin{proof}
Choose finite mixtures
\[
 \rho_n=\sum_i\lambda_{i,n}Q_{i,n}^{\otimes2}
\]
converging weakly to $\rho$, as permitted by \eqref{eq:C3-def}, and put $P_n:=\sum_i\lambda_{i,n}Q_{i,n}$.  The first marginal of $\rho_n$ is $P_n$, so $P_n$ converges weakly to $P$.  For every $n$,
\[
 \iint\psi(x)\psi(y)\,\rho_n(\dd x,\dd y)
 =\sum_i\lambda_{i,n}\left(\int\psi\,\dd Q_{i,n}\right)^2
 \ge\left(\int\psi\,\dd P_n\right)^2
\]
by Jensen's inequality.  Since $\psi$ is continuous and bounded on $[0,1]$, weak convergence allows passage to the limit.
\end{proof}

\medskip
\noindent\textbf{Independent construction.}
The independent protocol is
\[
 \Pi^{(0)}_{s,t}:=\operatorname{Ber}(s)\otimes\operatorname{Ber}(t),
 \qquad
 \Pi^{(0)}_{s,t}(0,0)=(1-s)(1-t).
\]

\medskip
\noindent\textbf{Parameterized Example~5.}
Liu's Example~5~\cite[Example~5]{liu} is specified by a measurable function $f:[0,1]\to\mathbb R$ satisfying
\[
 0\le f(x)\le\min\{x,1-x\}.
\]
For $0\le\ell\le1$, the choice
\[
 f_\ell(x):=\ell x(1-x)
\]
is admissible because $x(1-x)\le\min\{x,1-x\}$.  In complemented coordinates $x=1-s$, $y=1-t$, its entropy integrand is
\[
 H(r_\ell(x,y)),
\]
where
\begin{equation}
 r_\ell(x,y):=
 xy+\ell^2x(1-x)y(1-y)
 =xy\bigl(1+\ell^2(1-x)(1-y)\bigr),
 \label{eq:r-ell}
\end{equation}
and
\begin{equation}
 g_\ell(x):=r_\ell(x,x)
 =x^2\bigl(1+\ell^2(1-x)^2\bigr).
 \label{eq:g-ell}
\end{equation}
For $\ell=1$ we write
\[
 r:=r_1,\qquad g:=g_1.
\]

For a probability measure $P$ on $[0,1]$, write
\[
 m(P):=\int x\,P(\dd x).
\]
If $S$ has distribution $\mu$ and $P$ is the distribution of $1-S$, then
\begin{equation}
 m(P)=\int x\,P(\dd x)
 =1-\int s\,\mu(\dd s).
 \label{eq:mean-reflection}
\end{equation}
Thus $\int s\,\mu(\dd s)\le c$ is equivalent to $m(P)\ge1-c$, and
\[
 H\bigl(\Pi^{(0)}_{s,t}(0,0)\bigr)=H(xy).
\]

\section{Liu's conjectured threshold}

The goal of this section is to reduce Theorem~\ref{thm:liu-bound} to two pointwise inequalities for the independent protocol and Example~5 at $\ell=1$.  Sections~4 and~5 prove these inequalities.

Define
\begin{equation}
 m_*:=1-c_{\mathrm L}
 =\frac{x_*H(x_*)}{H(x_*^2)}.
 \label{eq:mstar}
\end{equation}
Since
\[
 g(x)=2x^2-2x^3+x^4
\]
and $p(x_*)=0$, we have $x_*^2+g(x_*)=1$.  Hence
\begin{equation}
 H(g(x_*))=H(x_*^2).
 \label{eq:H-equality-star}
\end{equation}

Choose $\alpha+\beta=1$ so that, on the diagonal $x=y$, the two sides of
\[
 m_*\bigl[\alpha H(x^2)+\beta H(g(x))\bigr]\ge xH(x)
\]
have the same derivative at $x=x_*$.  Equality of the values at $x_*$ follows from \eqref{eq:mstar} and \eqref{eq:H-equality-star}.  With
\[
 A_0(x):=H(x^2),\qquad B_0(x):=H(g(x)),
\]
the derivative condition gives
\begin{equation}
 \beta:=
 \frac{
 A_0(x_*)\left(\dfrac1{x_*}+\dfrac{H'(x_*)}{H(x_*)}\right)-A_0'(x_*)
 }{B_0'(x_*)-A_0'(x_*)},
 \qquad
 \alpha:=1-\beta.
 \label{eq:alpha-beta}
\end{equation}
Proposition~\ref{prop:two-construction-certificate} gives $\alpha,\beta\in(0,1)$.

\begin{lemma}
\label{lem:R}
For every $x,y\in[0,1]$,
\begin{equation}
 H(r(x,y))^2\ge H(g(x))H(g(y)).
 \label{eq:R}
\end{equation}
\end{lemma}

\begin{lemma}
\label{lem:S}
For every $x,y\in[0,1]$,
\begin{equation}
 2\alpha H(xy)+2\beta\sqrt{H(g(x))H(g(y))}
 \ge\frac{xH(y)+yH(x)}{m_*}.
 \label{eq:S}
\end{equation}
\end{lemma}

Put
\[
 b(x):=\sqrt{H(g(x))}.
\]
Lemma~\ref{lem:R}, Lemma~\ref{lem:positive-correlation}, and Lemma~\ref{lem:S} give the estimate for a prescribed marginal.

\begin{proposition}
\label{prop:fixed-two-construction}
Let $P$ be a probability measure on $[0,1]$ and let $\rho\in\cC_3(P)$.  Then
\begin{equation}
 \alpha\iint H(xy)\,P(\dd x)P(\dd y)
 +\beta\iint H(r(x,y))\,\rho(\dd x,\dd y)
 \ge \frac{m(P)}{m_*}\int H(x)\,P(\dd x).
 \label{eq:fixed-two-construction}
\end{equation}
\end{proposition}

\begin{proof}
Taking square roots in \eqref{eq:R} and applying Lemma~\ref{lem:positive-correlation} gives
\[
 \iint H(r(x,y))\,\rho(\dd x,\dd y)
 \ge\left(\int b\,\dd P\right)^2.
\]
Next integrate \eqref{eq:S} with respect to $P(\dd x)P(\dd y)$.  The right-hand side becomes, by Fubini,
\[
 \frac{2m(P)}{m_*}\int H\,\dd P.
\]
Dividing by $2$ gives
\[
 \alpha\iint H(xy)\,\dd P\dd P
 +\beta\left(\int b\,\dd P\right)^2
 \ge\frac{m(P)}{m_*}\int H\,\dd P.
\]
Replacing the squared integral by the preceding lower bound for the Example~5 integral proves \eqref{eq:fixed-two-construction}.
\end{proof}

\begin{proof}[Proof of Theorem~\ref{thm:liu-bound}]
Fix $0<c<c_{\mathrm L}$ and let $\mu$ be a probability measure on $[0,1]$ with $\int s\,\mu(\dd s)\le c$.  Let $S$ be a $[0,1]$-valued random variable with distribution $\mu$, and let $P$ be the distribution of $1-S$.  By \eqref{eq:mean-reflection},
\[
 m(P)\ge1-c>m_*.
\]
Set
\[
 \gamma:=\frac{1-c}{m_*}>1.
\]
Apply Proposition~\ref{prop:liu-criterion} to the independent protocol and Example~5 with weights $\alpha,\beta$, taking
\[
 \mathcal K_0(\mu)=\{\mu\otimes\mu\},
 \qquad
 \mathcal K_5(\mu)=\cC_3(\mu).
\]
For any $\rho\in\cC_3(\mu)$, reflect both coordinates: if $(X,Y)$ has joint law $\rho$, consider the law of $(1-X,1-Y)$.  Its two marginals are $P$, and it belongs to $\cC_3(P)$ by \eqref{eq:C3-def}: each product measure $Q^{\otimes2}$ is sent to the square of the one-dimensional law of $1-X$ when $X$ has distribution $Q$, and convex combinations and weak limits are preserved.  Moreover $H(1-s)=H(s)$.  Proposition~\ref{prop:fixed-two-construction} therefore bounds the left-hand side in Liu's criterion from below by
\[
 \frac{m(P)}{m_*}\int H\,\dd P
 \ge\gamma\int H\,\dd\mu.
\]
Thus $c\le c_{\mathrm{UC}}$.  Letting $c\uparrow c_{\mathrm L}$ proves the theorem.
\end{proof}

\section{The first pointwise inequality at Liu's parameter}

The goal of this section is to prove Lemma~\ref{lem:R}.  The proof splits according to the value of $g(y)$.  When $g(y)<1/2$, Lemmas~\ref{lem:entropy-one-variable}, \ref{lem:r-identity}, and~\ref{lem:WUV} show that a mixed derivative of $\log H(r(x,y))$ is negative.  When $g(y)\ge 1/2$, Lemma~\ref{lem:M-decreasing} gives the needed comparison directly.  We establish these four ingredients first and then prove Lemma~\ref{lem:R}.

For $u\in(0,1)$, where $H(u)>0$, set
\[
 \Lambda(u):=\log H(u),
 \qquad
 M(u):=\frac{uH'(u)}{H(u)}.
\]

The following one-variable inequality for $H$ will be used in the case $g(y)<1/2$.

\begin{lemma}
\label{lem:entropy-one-variable}
For $0<u<1/2$,
\begin{equation}
 -\frac{\Lambda''(u)}{\Lambda'(u)}
 >\frac1u+\frac1{1-u}.
 \label{eq:entropy-one-variable}
\end{equation}
\end{lemma}

\begin{proof}
For $0<u<1/2$, put
\[
 \lambda:=\log\frac{1-u}{u}>0,
 \qquad
 q:=-\log(1-u)>0.
\]
Then
\[
 H'(u)=\lambda,
 \qquad
 H(u)=u\lambda+q,
\]
and therefore
\begin{equation}
 M(u)=\frac{u\lambda}{u\lambda+q}.
 \label{eq:M-ul}
\end{equation}
For $0<u<1/2$, both $M(u)$ and $\Lambda'(u)=H'(u)/H(u)$ are positive.  Since $M=u\Lambda'$, we may take logarithms:
\[
 \log M=\log u+\log\Lambda'.
\]
Differentiating term by term gives
\[
 \frac{M'}M=\frac1u+\frac{\Lambda''}{\Lambda'}.
\]
Hence
\[
 -\frac{\Lambda''}{\Lambda'}=-\frac{M'}M+\frac1u,
\]
and therefore
\[
 -\frac{\Lambda''}{\Lambda'}>\frac1u+\frac1{1-u}
 \quad\Longleftrightarrow\quad
 -\frac{M'}M>\frac1{1-u}.
\]
Thus it remains to prove
\begin{equation}
 -\frac{M'}M>\frac1{1-u}.
 \label{eq:M-target}
\end{equation}

We compute the left-hand side explicitly.  Since
\[
 \lambda'=-\frac1{u(1-u)},
\]
taking logarithms in \eqref{eq:M-ul} gives
\[
 \log M=\log u+\log\lambda-\log H(u).
\]
Differentiating term by term gives
\[
 \frac{M'}M
 =\frac1u+\frac{\lambda'}\lambda-\frac{H'(u)}{H(u)}
 =\frac1u-\frac1{u(1-u)\lambda}-\frac{\lambda}{H(u)}.
\]
Hence
\begin{align}
 -\frac{M'}M-\frac1{1-u}
 &=-\frac1u+\frac1{u(1-u)\lambda}+\frac{\lambda}{H(u)}-\frac1{1-u}\nonumber\\
 &=\frac{1-\lambda}{u(1-u)\lambda}+\frac{\lambda}{H(u)}\nonumber\\
 &=\frac{H(u)(1-\lambda)+u(1-u)\lambda^2}
 {u(1-u)\lambda H(u)}\nonumber\\
 &=\frac{q(1-\lambda)+u\lambda(1-u\lambda)}
 {u(1-u)\lambda H(u)}.
 \label{eq:M-numerator}
\end{align}
The denominator in \eqref{eq:M-numerator} is positive, so it remains to show that the numerator is positive.

If $0<\lambda\le1$, then $q(1-\lambda)\ge0$.  Moreover $0<u\lambda\le u<1/2$, so $u\lambda(1-u\lambda)>0$.  Thus the numerator is positive.

Assume now $\lambda>1$.  The elementary estimate
\[
 -\log(1-u)\le\frac{u}{1-u}
\]
gives $q\le u/(1-u)$.  Because $1-\lambda<0$, multiplying the inequality $q\le u/(1-u)$ by $1-\lambda$ reverses the inequality sign and gives
\[
 q(1-\lambda)\ge\frac{u}{1-u}(1-\lambda).
\]
Therefore
\begin{align}
 q(1-\lambda)+u\lambda(1-u\lambda)
 &\ge \frac{u}{1-u}
 \left(1-u\lambda(\lambda+1)+u^2\lambda^2\right).
 \label{eq:lambda-lower}
\end{align}
Now
\[
 u=\frac1{1+e^\lambda}<e^{-\lambda}.
\]
For $\lambda\ge1$ we have $e^\lambda>\lambda(\lambda+1)$.  To see this without any approximation, use the cubic Taylor polynomial:
\[
 e^\lambda>1+\lambda+\frac{\lambda^2}{2}+\frac{\lambda^3}{6}.
\]
Subtracting $\lambda(\lambda+1)$ gives
\[
 1-\frac{\lambda^2}{2}+\frac{\lambda^3}{6}.
\]
Its derivative is $\lambda(\lambda/2-1)$, so its minimum on $[1,\infty)$ is attained at $\lambda=2$ and equals $1/3$.  Hence it is positive.  Consequently
\[
 u\lambda(\lambda+1)<1.
\]
The parenthesized factor in \eqref{eq:lambda-lower} is therefore strictly positive: the sum of its first two terms is $1-u\lambda(\lambda+1)>0$, and its last term $u^2\lambda^2$ is nonnegative.  This proves \eqref{eq:M-target}, hence \eqref{eq:entropy-one-variable}.
\end{proof}

Specializing \eqref{eq:r-ell} to $\ell=1$,
\[
 r(x,y)=xy\bigl(2-x-y+xy\bigr).
\]
Define
\[
 W:=2-x-y+xy,
 \qquad
 U:=2-2x-y+2xy,
 \qquad
 V:=2-x-2y+2xy.
\]
Thus $r=xyW$, and direct differentiation gives
\begin{equation}
 r_x=yU,\qquad r_y=xV,\qquad r_{xy}=2-2x-2y+4xy.
 \label{eq:r-derivatives}
\end{equation}
For $0<x,y<1$, one has $W=1+(1-x)(1-y)>0$, $U=y+2(1-x)(1-y)>0$, and $V=x+2(1-x)(1-y)>0$.  Hence $r$ is strictly increasing in each variable on $(0,1)^2$, and in particular $0<r(x,y)<1$ there.

The next identity rewrites the mixed derivative of $r$ in the form needed to combine it with Lemma~\ref{lem:entropy-one-variable}.

\begin{lemma}
\label{lem:r-identity}
For $0<x,y<1$, with $W,U,V$ defined above,
\begin{equation}
 \frac{r_{xy}}{r_xr_y}
 =\frac1r+\frac1{WUV}.
 \label{eq:r-identity}
\end{equation}
\end{lemma}

\begin{proof}
Substitute $r=xyW$, $r_x=yU$, and $r_y=xV$.  By the positivity just established, multiplication by $xyWUV$ preserves the equality and reduces \eqref{eq:r-identity} to
\[
 W\,r_{xy}=UV+xy.
\]
Using the definitions of $W,U,V$ and \eqref{eq:r-derivatives},
\begin{align*}
 W\,r_{xy}
 &=(2-x-y+xy)(2-2x-2y+4xy),\\
 UV+xy
 &=(2-2x-y+2xy)(2-x-2y+2xy)+xy.
\end{align*}
Both sides expand to
\[
4-6x-6y+2x^2+2y^2+14xy
-6x^2y-6xy^2+4x^2y^2.
\]
\end{proof}

Let $y_0$ be the unique number in $(0,1)$ with $g(y_0)=1/2$.  Its existence follows from $g(0)=0<1/2<1=g(1)$ and continuity.  Indeed,
\[
 g'(x)=4x-6x^2+4x^3=2x(2-3x+2x^2)>0
 \qquad(0<x<1),
\]
so $g$ is strictly increasing.  The calculation
\[
 g(43/64)-\frac12=\frac{305}{2^{24}}>0
\]
shows that
\begin{equation*}
 y_0<\frac{43}{64}.
\end{equation*}

The following polynomial inequality allows us to compare the second term in Lemma~\ref{lem:r-identity} with the right-hand side of Lemma~\ref{lem:entropy-one-variable} on the square containing $[0,y_0]^2$.

\begin{lemma}
\label{lem:WUV}
For $0\le x,y\le43/64$,
\begin{equation}
 WUV\ge1-r(x,y).
 \label{eq:WUV}
\end{equation}
\end{lemma}

\begin{proof}
After the rescaling
\[
 x=\frac{43}{64}s,
 \qquad
 y=\frac{43}{64}t,
 \qquad
 0\le s,t\le1,
\]
the polynomial $WUV-(1-r)$ has bidegree $(3,3)$. Proposition~\ref{prop:bernstein-certificate} gives its exact coefficient matrix in the product Bernstein basis of degree $(3,3)$; every coefficient is strictly positive. Since every Bernstein basis function is nonnegative on $[0,1]^2$, the polynomial is nonnegative there. Rescaling back proves \eqref{eq:WUV}.
\end{proof}

For the case $g(y)\ge 1/2$, define
\[
 G(v):=e^vH(e^{-v})\qquad(v>0).
\]
The following monotonicity property of $M$ handles this case.

\begin{lemma}
\label{lem:M-decreasing}
The function $M(u)=uH'(u)/H(u)$ is strictly decreasing on $(0,1)$.
\end{lemma}

\begin{proof}
Put $u=e^{-v}$.  By the definition above,
\[
 G(v)=\frac{H(u)}u,
 \qquad
 H(e^{-v})=e^{-v}G(v).
\]
Hence
\[
 \log H(e^{-v})=-v+\log G(v).
\]
Using $-\log(1-z)=\sum_{k\ge1}z^k/k$ with $z=e^{-v}$ gives
\[
 G(v)=v+1-\sum_{k\ge1}\frac{e^{-kv}}{k(k+1)}.
\]
Termwise differentiation yields
\[
 G''(v)=-\sum_{k\ge1}\frac{k e^{-kv}}{k+1}<0.
\]
Since $G(v)>0$,
\[
 (\log G)''=\frac{G''}{G}-\left(\frac{G'}G\right)^2<0:
\]
the first term is strictly negative and the second is nonpositive.  Thus $\log G$ is strictly concave.  The function $v\mapsto-v$ is affine, and adding an affine function preserves strict concavity; hence $v\mapsto\log H(e^{-v})=-v+\log G(v)$ is strictly concave.

Its derivatives can be computed directly.  Since $u=e^{-v}$ and $u'=-u$,
\[
 \frac{\dd}{\dd v}\log H(e^{-v})
 =\frac{H'(u)}{H(u)}(-u)
 =-\frac{uH'(u)}{H(u)}
 =-M(u).
\]
Differentiating once more,
\[
 \frac{\dd^2}{\dd v^2}\log H(e^{-v})
 =-M'(u)u'
 =uM'(u)
 =e^{-v}M'(e^{-v}).
\]
Strict concavity makes this quantity negative.  Since $e^{-v}>0$, we obtain $M'(e^{-v})<0$, and therefore $M'(u)<0$ for every $0<u<1$.

\end{proof}

\begin{proof}[Proof of Lemma~\ref{lem:R}]
Both sides of \eqref{eq:R} are continuous on $[0,1]^2$, so the boundary inequality follows once the inequality is proved in the open square. On the diagonal, $g(x)=r(x,x)$ gives equality. Thus it remains only to prove the strict inequality for $x\ne y$ in the open square. By symmetry assume $0<x<y<1$ and define
\[
 \mathcal L(x,y):=\log H(r(x,y))
\]
and
\begin{equation}
 \Delta_x(y):=\mathcal L(x,y)-\frac12\mathcal L(x,x)-\frac12\mathcal L(y,y).
 \label{eq:Deltax-def}
\end{equation}
Since $g(z)=r(z,z)$, the strict inequality in Lemma~\ref{lem:R} is
\[
 2\mathcal L(x,y)>\mathcal L(x,x)+\mathcal L(y,y),
\]
which is exactly $\Delta_x(y)>0$.  Moreover $\Delta_x(x)=0$.  Thus it suffices to prove $\Delta_x'(s)>0$ for every $s\in(x,y]$, because then
\[
 \Delta_x(y)=\Delta_x(x)+\int_x^y\Delta_x'(s)\,\dd s>0.
\]
After renaming $s$ as $y$, the remaining goal is therefore
\[
 \Delta_x'(y)>0\qquad(0<x<y<1).
\]

Differentiate \eqref{eq:Deltax-def} with respect to $y$.  The term $\mathcal L(x,x)$ is constant, while symmetry gives $\mathcal L_1(y,y)=\mathcal L_2(y,y)$.  Hence
\begin{equation}
 \Delta_x'(y)=\mathcal L_2(x,y)-\mathcal L_2(y,y).
 \label{eq:Deltax-derivative}
\end{equation}
We split the proof into two cases according to the value of $g(y)$.

\smallskip
\noindent\emph{Case A: $g(y)<1/2$.}  Then $y<y_0$, and for every $s\in[x,y]$ we have $s,y\le y_0$ and
\[
 r(s,y)\le r(y,y)=g(y)<1/2.
\]
By definition, $\mathcal L(x,y)=\log H(r(x,y))=\Lambda(r(x,y))$, so $\mathcal L=\Lambda\circ r$.  The chain rule gives
\[
 \mathcal L_{12}=\Lambda''(r)r_xr_y+\Lambda'(r)r_{xy}
 =\Lambda'(r)r_xr_y
 \left(\frac{r_{xy}}{r_xr_y}+\frac{\Lambda''(r)}{\Lambda'(r)}\right).
\]
Here $\Lambda'(r)=H'(r)/H(r)>0$ because $r<1/2$, and $r_x,r_y>0$.  By Lemmas~\ref{lem:r-identity} and~\ref{lem:WUV},
\[
 \frac{r_{xy}}{r_xr_y}
 =\frac1r+\frac1{WUV}
 \le\frac1r+\frac1{1-r}.
\]
Lemma~\ref{lem:entropy-one-variable} gives
\[
 \frac{\Lambda''(r)}{\Lambda'(r)}
 <-\frac1r-\frac1{1-r}.
\]
Therefore $\mathcal L_{12}(s,y)<0$ for every $s\in[x,y]$.  Apply the fundamental theorem of calculus to the function $s\mapsto\mathcal L_2(s,y)$:
\[
 \mathcal L_2(y,y)-\mathcal L_2(x,y)=\int_x^y\mathcal L_{12}(s,y)\,\dd s<0.
\]
Thus $\mathcal L_2(x,y)>\mathcal L_2(y,y)$, and equation~\eqref{eq:Deltax-derivative} gives $\Delta_x'(y)>0$.

\smallskip
\noindent\emph{Case B: $g(y)\ge1/2$.}  By Lemma~\ref{lem:M-decreasing}, $M$ is strictly decreasing on $(0,1)$.

Next, differentiating $\mathcal L$ with respect to its second variable gives
\begin{equation}
 \mathcal L_2(x,y)=\frac{r_y(x,y)}{r(x,y)}M(r(x,y))
 =A_x(y)M(r(x,y)),
 \label{eq:K2-M}
\end{equation}
where
\begin{equation}
 A_x(y):=\frac1y-
 \frac{1-x}{1+(1-x)(1-y)}.
 \label{eq:Ax-def}
\end{equation}
For $0<x<y<1$, the derivative $r_y$ is positive, so $A_x(y)>0$.  Moreover the function
\[
 u\longmapsto\frac{u}{1+u(1-y)}
\]
is strictly increasing for $u\ge0$.  Since $1-x>1-y$, equation \eqref{eq:Ax-def} gives
\begin{equation}
 0<A_x(y)<A_y(y).
 \label{eq:A-order}
\end{equation}
Also $r(x,y)<r(y,y)=g(y)$ because $r$ is strictly increasing in its first variable.

If $r(x,y)<1/2\le g(y)$, then
\[
 M(r(x,y))>0\ge M(g(y)),
\]
so combining \eqref{eq:K2-M} with \eqref{eq:A-order} gives
\[
 \mathcal L_2(x,y)>\mathcal L_2(y,y).
\]
If $1/2\le r(x,y)<g(y)$, then the strict decrease of $M$ gives
\[
 M(g(y))<M(r(x,y))\le0.
\]
Therefore
\[
 A_x(y)M(r(x,y))
 >A_x(y)M(g(y))
 >A_y(y)M(g(y)),
\]
where the second inequality reverses because $M(g(y))<0$ and $A_x(y)<A_y(y)$.  Again $\mathcal L_2(x,y)>\mathcal L_2(y,y)$.  Thus $\Delta_x'(y)>0$ in every case.

Both cases give $\Delta_x'(y)>0$. Hence the integral identity following \eqref{eq:Deltax-def} gives $\Delta_x(y)>0$, which proves \eqref{eq:R} in the open square.  Continuity gives the boundary inequality.
\end{proof}

\section{The second pointwise inequality at Liu's parameter}

The goal of this section is to prove Lemma~\ref{lem:S}. We first rewrite its inequality as $F(t,d)\ge0$. Two analytic lemmas prove positivity in three boundary regions. We then treat the zero forced by the diagonal tangency and use the computational results of Appendix~A on the compact part left over. These regions cover the entire domain $t>0$, $0\le d\le t$.

By symmetry it is enough to consider $0<x\le y<1$. Introduce logarithmic coordinates
\begin{equation}
 x=e^{-(t+d)},
 \qquad
 y=e^{-(t-d)},
 \qquad
 t>0,\quad0\le d<t.
 \label{eq:td-coordinates}
\end{equation}
Recall from Section~4 that
\[
 G(u)=e^uH(e^{-u})=\frac{H(e^{-u})}{e^{-u}}
 \qquad(u>0),
\]
and define
\[
 \Gamma_5(u):=\frac{\sqrt{H(g(e^{-u}))}}{e^{-u}}.
\]
Both functions have continuous extensions at $u=0$, with $G(0)=\Gamma_5(0)=0$. Define
\begin{equation}
 F(t,d):=
 m_*\bigl[\alpha G(2t)+\beta\Gamma_5(t+d)\Gamma_5(t-d)\bigr]
 -\frac{G(t+d)+G(t-d)}2.
 \label{eq:F-def}
\end{equation}

The next proposition records the reduction used throughout the section.

\begin{proposition}
\label{prop:F-reduction}
For $0<x\le y<1$, the change of variables \eqref{eq:td-coordinates} is a bijection with $t>0$, $0\le d<t$, and inequality \eqref{eq:S} is equivalent to $F(t,d)\ge0$. The formula \eqref{eq:F-def} extends continuously to $d=t$, so below we prove $F(t,d)\ge0$ on the larger domain $t>0$, $0\le d\le t$.
\end{proposition}

\begin{proof}
From \eqref{eq:td-coordinates}, $xy=e^{-2t}$ and
\[
 \frac{H(xy)}{xy}=G(2t).
\]
Moreover,
\[
 \frac{\sqrt{H(g(x))H(g(y))}}{xy}
 =\frac{\sqrt{H(g(e^{-(t+d)}))}}{e^{-(t+d)}}
  \frac{\sqrt{H(g(e^{-(t-d)}))}}{e^{-(t-d)}}
 =\Gamma_5(t+d)\Gamma_5(t-d),
\]
and
\[
 \frac{xH(y)+yH(x)}{2xy}
 =\frac12\left(\frac{H(y)}y+\frac{H(x)}x\right)
 =\frac{G(t-d)+G(t+d)}2.
\]
Multiplying \eqref{eq:S} by the positive factor $m_*/(2xy)$ gives
\begin{align*}
 &m_*\left[
 \alpha\frac{H(xy)}{xy}
 +\beta\frac{\sqrt{H(g(x))H(g(y))}}{xy}
 \right]
 -\frac12\left(\frac{H(y)}y+\frac{H(x)}x\right)\ge0.
\end{align*}
Substituting the three identities just proved turns the left-hand side into
\[
 m_*\bigl[\alpha G(2t)+\beta\Gamma_5(t+d)\Gamma_5(t-d)\bigr]
 -\frac{G(t+d)+G(t-d)}2
 =F(t,d).
\]
Thus \eqref{eq:S} is equivalent to $F(t,d)\ge0$.
\end{proof}

To obtain the two analytic boundary estimates, we need monotonicity, concavity, and elementary bounds for $G$.

\begin{lemma}
\label{lem:G-properties}
The function $G$ is increasing and concave on $[0,\infty)$, and strictly increasing and strictly concave on $(0,\infty)$.  For every $u\ge0$,
\[
 u+1-e^{-u}\le G(u)\le u+1.
\]
\end{lemma}

\begin{proof}
Starting from the definition of $H$ with $x=e^{-u}$,
\[
 G(u)=\frac{-x\log x-(1-x)\log(1-x)}x
 =u+(e^u-1)\bigl[-\log(1-e^{-u})\bigr].
\]
Using $-\log(1-z)=\sum_{k\ge1}z^k/k$ at $z=e^{-u}$ and multiplying by $e^u-1$ gives
\begin{equation}
 G(u)=u+1-\sum_{k\ge1}\frac{e^{-ku}}{k(k+1)}.
 \label{eq:G-series}
\end{equation}
The series and its first two derivatives converge locally uniformly for $u>0$, so termwise differentiation is valid there.  We obtain
\[
 G'(u)=1+\sum_{k\ge1}\frac{e^{-ku}}{k+1}>0,
 \qquad
 G''(u)=-\sum_{k\ge1}\frac{k e^{-ku}}{k+1}<0.
\]
This proves strict monotonicity and strict concavity on $(0,\infty)$.  Together with $G(0)=0$, these properties extend in the non-strict form to $[0,\infty)$.

Finally, $e^{-ku}\le e^{-u}$ for $k\ge1$, and
\[
 \sum_{k\ge1}\frac1{k(k+1)}
 =\sum_{k\ge1}\left(\frac1k-\frac1{k+1}\right)=1.
\]
Hence
\[
 0\le\sum_{k\ge1}\frac{e^{-ku}}{k(k+1)}
 \le e^{-u}\sum_{k\ge1}\frac1{k(k+1)}=e^{-u}.
\]
Writing the series in \eqref{eq:G-series} as $S(u)$, the estimate just obtained is $0\le S(u)\le e^{-u}$. Hence
\[
 u+1-e^{-u}\le u+1-S(u)=G(u)\le u+1,
\]
which gives the two stated bounds.
\end{proof}

We first prove strict positivity when $t$ is small.

\begin{lemma}
\label{lem:F-small-t}
For $0<t\le0.002$ and $0\le d\le t$ one has $F(t,d)>0$.
\end{lemma}

\begin{proof}
Put
\[
 c_0:=m_*\alpha.
\]
The interval bounds in Proposition~\ref{prop:two-construction-certificate} imply
\[
 c_0>0.5555293761467>\frac12.
\]
The two arguments $t+d$ and $t-d$ have average $t$.  Since $G$ is concave by Lemma~\ref{lem:G-properties},
\[
 \frac{G(t+d)+G(t-d)}2\le G(t).
\]
The factor $\beta$ and the function $\Gamma_5$ are nonnegative, so the corresponding term in \eqref{eq:F-def} can be discarded in a lower bound.  Hence
\begin{equation}
 F(t,d)\ge c_0G(2t)-G(t).
 \label{eq:smallt-first}
\end{equation}
We estimate the two values of $G$ explicitly.

The elementary inequality $e^t-1\le te^t$ and the lower bound $e^t-1\ge t$ imply
\[
 1-e^{-t}=e^{-t}(e^t-1)\ge te^{-t}.
\]
Since $-\log$ is decreasing on $(0,\infty)$, from $1-e^{-t}\ge te^{-t}>0$ we get
\[
 -\log(1-e^{-t})\le-\log(te^{-t})=-\log t+t.
\]
Using the formula for $G$ at the start of Lemma~\ref{lem:G-properties},
\begin{equation}
 G(t)\le t+te^t(-\log t+t).
 \label{eq:G-small-upper}
\end{equation}
For $2t$, we use $e^{2t}-1\ge2t$ and $1-e^{-2t}\le2t$.  The latter gives $-\log(1-e^{-2t})\ge-\log(2t)$, and therefore
\begin{equation}
 G(2t)\ge2t+2t[-\log(2t)]=2t(1-\log(2t)).
 \label{eq:G-small-lower}
\end{equation}
Define
\[
 \zeta(t):=2c_0(1-\log(2t))-1-e^t(-\log t+t).
\]
Substituting \eqref{eq:G-small-upper} and \eqref{eq:G-small-lower} into \eqref{eq:smallt-first} gives
\[
 F(t,d)\ge t\,\zeta(t).
\]
Differentiation gives
\[
 \zeta'(t)=\frac{e^t-2c_0}{t}-e^t(1-\log t+t).
\]
For $0<t\le0.002$, we have $e^t<1.003<2c_0$.  Thus the first term is negative.  The second is also negative because $1-\log t+t>0$.  Hence $\zeta'(t)<0$ throughout the interval, so $\zeta$ is minimized at $t=0.002$.  Proposition~\ref{prop:two-construction-certificate} gives the certified lower bound
\[
 \zeta(0.002)\ge0.016672474402352748>0.
\]
Since $t>0$, the inequality $F(t,d)\ge t\zeta(t)$ now proves the claim.
\end{proof}

The next estimate covers points close to the edge $d=t$ and all sufficiently large $t$.

\begin{lemma}
\label{lem:F-boundary-tail}
Let $v=t-d$. Then $F(t,d)>0$ whenever one of the following holds:
\begin{enumerate}
 \item[(a)] $0.002\le t\le1.4$ and $0\le v\le5\cdot10^{-5}$;
 \item[(b)] $1.4\le t\le4.1$ and $0\le v\le0.1$;
 \item[(c)] $t\ge4.1$ and $0\le d\le t$.
\end{enumerate}
\end{lemma}

\begin{proof}
Set
\[
 u:=t+d,
 \qquad
 v:=t-d.
\]
Then $u+v=2t$ and $0\le v\le t\le u\le2t$.  From \eqref{eq:F-def},
\[
 F(t,d)=c_0G(2t)
 +m_*\beta\Gamma_5(u)\Gamma_5(v)
 -\frac12G(u)-\frac12G(v).
\]
Since $m_*,\beta,\Gamma_5(u),\Gamma_5(v)$ are nonnegative, deleting their product gives
\[
 F(t,d)\ge c_0G(2t)-\frac12G(u)-\frac12G(v).
\]
Also $u\le2t$ and $G$ is increasing, so $G(u)\le G(2t)$.  Therefore
\[
 F(t,d)
 \ge\left(c_0-\frac12\right)G(2t)-\frac12G(v).
\]
The last expression is increasing in $t$ and decreasing in $v$, because $c_0>1/2$ and $G$ is increasing.  Thus its minimum on each of the two rectangles occurs at the smallest allowed $t$ and the largest allowed $v$.  The corresponding certified interval evaluations, recorded in Proposition~\ref{prop:two-construction-certificate}, are
\[
 \left(c_0-\frac12\right)G(0.004)-\frac12G(5\cdot10^{-5})
 \ge0.0011788376979667077
\]
and
\[
 \left(c_0-\frac12\right)G(2.8)-\frac12G(0.1)
 \ge0.035598095197518845.
\]
The two certified endpoint evaluations prove parts~(a) and~(b).

For part~(c), Lemma~\ref{lem:G-properties} gives
\[
 G(2t)\ge2t+1-e^{-2t}.
\]
It also gives $G(u)\le u+1$ and $G(v)\le v+1$, so, using $u+v=2t$,
\[
 \frac{G(u)+G(v)}2
 \le\frac{(u+1)+(v+1)}2=t+1.
\]
Consequently
\begin{equation}
 F(t,d)\ge c_0(2t+1-e^{-2t})-(t+1).
 \label{eq:tail-bound}
\end{equation}
The derivative of the lower bound in \eqref{eq:tail-bound} is
\[
 2c_0+2c_0e^{-2t}-1>2c_0-1>0,
\]
so this lower bound is increasing in $t$.  At the left endpoint $t=4.1$,
\[
 2t+1-e^{-2t}=9.2-e^{-8.2},
 \qquad
 t+1=5.1.
\]
Thus Proposition~\ref{prop:two-construction-certificate} gives
\[
 c_0(9.2-e^{-8.2})-5.1\ge0.010717682423498683>0.
\]
Hence \eqref{eq:tail-bound} is positive for every $t\ge4.1$, proving part~(c).
\end{proof}

To treat the interior, we first identify the stationary point corresponding to the diagonal tangency at $x=y=x_*$.

Let
\[
 t_*:=-\log x_*.
\]

\begin{lemma}
\label{lem:F-stationary}
The point $(t_*,0)$ satisfies:
\begin{enumerate}
\item[(a)] $F(t_*,0)=0$;
\item[(b)] $F_t(t_*,0)=0$;
\item[(c)] $F_d(t_*,0)=0$.
\end{enumerate}
\end{lemma}

\begin{proof}
For part~(c), the function $F(t,d)$ is even in $d$, as is clear from \eqref{eq:F-def}; hence $F_d(t,0)=0$ and in particular $F_d(t_*,0)=0$.

For parts~(a) and~(b), put $x=e^{-t}$ and define
\begin{equation*}
 \Delta_0(x):=m_*\bigl[\alpha H(x^2)+\beta H(g(x))\bigr]-xH(x).
\end{equation*}
Setting $d=0$ in \eqref{eq:F-def} and using $G(t)=H(x)/x$ and $\Gamma_5(t)^2=H(g(x))/x^2$ gives
\begin{equation}
 F(t,0)=e^{2t}\Delta_0(e^{-t}).
 \label{eq:F-Delta0}
\end{equation}
At $x=x_*$, equation~\eqref{eq:H-equality-star}, the identity $\alpha+\beta=1$, and the definition of $m_*$ give
\[
 \Delta_0(x_*)=m_*H(x_*^2)-x_*H(x_*)=0.
\]
Thus $F(t_*,0)=0$, proving part~(a).

To prove part~(b), differentiate \eqref{eq:F-Delta0}. Writing $x=e^{-t}$, so $x'=-x$, gives
\[
 F_t(t,0)=e^{2t}\bigl(2\Delta_0(x)-x\Delta_0'(x)\bigr).
\]
At $t=t_*$ the first term vanishes because $\Delta_0(x_*)=0$.  Since $e^{2t_*}x_*>0$, we have
\[
 F_t(t_*,0)=0
 \quad\Longleftrightarrow\quad
 \Delta_0'(x_*)=0.
\]
Finally,
\[
 \Delta_0'(x)=m_*\bigl[\alpha A_0'(x)+\beta B_0'(x)\bigr]-H(x)-xH'(x).
\]
At $x=x_*$, substitute $m_*=x_*H(x_*)/A_0(x_*)$ and $\alpha=1-\beta$.  After division by the positive factor $x_*H(x_*)/A_0(x_*)$, the equation $\Delta_0'(x_*)=0$ becomes
\[
 A_0'(x_*)+\beta\bigl(B_0'(x_*)-A_0'(x_*)\bigr)
 =A_0(x_*)\left(\frac1{x_*}+\frac{H'(x_*)}{H(x_*)}\right),
\]
which is exactly the defining equation \eqref{eq:alpha-beta} for $\beta$.  Hence $F_t(t_*,0)=0$, proving part~(b).
\end{proof}

The two boundary lemmas leave a compact region.  The next result combines the analytic estimates above with Proposition~\ref{prop:two-construction-certificate} to cover it, including the point $(t_*,0)$ where $F$ vanishes.

\begin{lemma}
\label{lem:F-global}
For every $t>0$ and $0\le d\le t$,
\[
 F(t,d)\ge0.
\]
\end{lemma}

\begin{proof}
To apply the computational results of Proposition~\ref{prop:two-construction-certificate}, we first express the second derivatives of $F$ in terms of derivatives of $G$ and $\Gamma_5$.

We begin with $G$.  Put $x=e^{-u}$.  Then $x'=-x$, and
\[
 H'(x)=\log\frac{1-x}{x},
 \qquad
 H''(x)=-\frac1{x(1-x)}.
\]
Since $G(u)=e^uH(x)$, the product and chain rules give
\[
 G'(u)=e^uH(x)+e^uH'(x)(-x)=G(u)-H'(x).
\]
Differentiating again,
\[
 G''(u)=G'(u)-H''(x)x'
 =G'(u)-\frac1{1-x}.
\]

For $\Gamma_5$, keep $x=e^{-u}$ and put $z=g(x)$.  The chain rule gives
\[
 z'=-xg'(x),
 \qquad
 z''=xg'(x)+x^2g''(x).
\]
Since $\Gamma_5(u)=e^u\sqrt{H(z)}$, another application of the product and chain rules gives
\begin{equation*}
 \Gamma_5'(u)
 =\Gamma_5(u)\left(1+\frac{H'(z)z'}{2H(z)}\right).
\end{equation*}
For the second derivative, differentiate the parenthesized factor as well.  Because
\[
 \frac{\dd}{\dd u}\frac{H'(z)z'}{H(z)}
 =\frac{H''(z)(z')^2+H'(z)z''}{H(z)}
 -\left(\frac{H'(z)z'}{H(z)}\right)^2,
\]
we obtain
\begin{align*}
 \Gamma_5''(u)
 ={}&\Gamma_5(u)\Biggl[
 \left(1+\frac{H'(z)z'}{2H(z)}\right)^2 \\[-1mm]
 &\qquad+\frac12\left(
 \frac{H''(z)(z')^2+H'(z)z''}{H(z)}
 -\left(\frac{H'(z)z'}{H(z)}\right)^2
 \right)\Biggr].
\end{align*}

Finally, calculate the second derivatives of $F$.  With $u=t+d$ and $v=t-d$, differentiating \eqref{eq:F-def} twice gives
\begin{align*}
 F_{tt}
 &=4m_*\alpha G''(2t)
 +m_*\beta\bigl(\Gamma_5''(u)\Gamma_5(v)+2\Gamma_5'(u)\Gamma_5'(v)+\Gamma_5(u)\Gamma_5''(v)\bigr)\\
 &\hspace{18mm}-\frac12\bigl(G''(u)+G''(v)\bigr),
 \\
 F_{dd}
 &=m_*\beta\bigl(\Gamma_5''(u)\Gamma_5(v)-2\Gamma_5'(u)\Gamma_5'(v)+\Gamma_5(u)\Gamma_5''(v)\bigr)\\
 &\hspace{18mm}-\frac12\bigl(G''(u)+G''(v)\bigr),
 \\
 F_{td}
 &=m_*\beta\bigl(\Gamma_5''(u)\Gamma_5(v)-\Gamma_5(u)\Gamma_5''(v)\bigr)
 -\frac12\bigl(G''(u)-G''(v)\bigr).
\end{align*}

Define the local rectangle
\begin{equation}
 \mathcal Q_0:=
 \bigl\{(t,d):|t-0.36992289186771087|\le0.06,\ 0\le d\le0.06\bigr\}.
 \label{eq:local-rectangle}
\end{equation}
Proposition~\ref{prop:two-construction-certificate} gives throughout $\mathcal Q_0$
\[
 F_{tt}\ge0.13459156779665421,
 \qquad
 F_{dd}\ge0.87181314699157286,
\]
and
\[
 F_{tt}F_{dd}-F_{td}^2\ge0.050114365218749776>0.
\]
Thus the Hessian is positive definite throughout $\mathcal Q_0$, and $F$ is strictly convex there.  The enclosure for $t_*$ in Proposition~\ref{prop:two-construction-certificate} lies inside the $t$-interval defining $\mathcal Q_0$, so $(t_*,0)\in\mathcal Q_0$.  By Lemma~\ref{lem:F-stationary}, this point is stationary and has value $0$.  Strict convexity therefore shows that $F\ge0$ throughout $\mathcal Q_0$.

It remains to cover precisely the part of the triangle not handled by Lemma~\ref{lem:F-small-t}, Lemma~\ref{lem:F-boundary-tail}, or the local rectangle.  To work with rectangular cells, set
\[
 \xi:=d/t\in[0,1],
 \qquad
 u=t(1+\xi),
 \qquad
 v=t(1-\xi),
\]
and define $\widetilde F(t,\xi):=F(t,t\xi)$.  The change $d=t\xi$ maps the triangular condition $0\le d\le t$ to the fixed interval $0\le\xi\le1$.

The first derivatives used in the finite estimates follow from the ordinary chain rule.  Since
\[
 \frac{\partial}{\partial t}F(t,t\xi)=F_t(t,t\xi)+\xi F_d(t,t\xi),
 \qquad
 \frac{\partial}{\partial\xi}F(t,t\xi)=tF_d(t,t\xi),
\]
substitution of \eqref{eq:F-def} gives
\begin{align*}
 \widetilde F_t
 &=2m_*\alpha G'(2t)
 +m_*\beta\bigl((1+\xi)\Gamma_5'(u)\Gamma_5(v)+(1-\xi)\Gamma_5(u)\Gamma_5'(v)\bigr)\\
 &\hspace{18mm}-\frac12\bigl((1+\xi)G'(u)+(1-\xi)G'(v)\bigr),
 \\
 \widetilde F_\xi
 &=m_*\beta t\bigl(\Gamma_5'(u)\Gamma_5(v)-\Gamma_5(u)\Gamma_5'(v)\bigr)
 -\frac t2\bigl(G'(u)-G'(v)\bigr).
\end{align*}

We call a \emph{box} a closed axis-parallel rectangle
\[
 R=[t_0-\Delta_t,t_0+\Delta_t]\times
 [\xi_0-\Delta_\xi,\xi_0+\Delta_\xi].
\]
For any $(t,\xi)\in R$, apply the one-variable fundamental theorem of calculus first along the $t$ direction and then along the $\xi$ direction.  Taking absolute values gives
\begin{equation}
 \widetilde F(t,\xi)
 \ge \widetilde F(t_0,\xi_0)
 -\sup_R|\widetilde F_t|\,\Delta_t
 -\sup_R|\widetilde F_\xi|\,\Delta_\xi.
 \label{eq:mean-value}
\end{equation}
Thus a lower bound for the center value and upper bounds for the two derivatives give a lower bound for $\widetilde F$ on the whole box.

After the analytic boundary regions are removed, the remaining points have $v=t(1-\xi)$ bounded below as follows:
\[
\begin{array}{c|c}
 t\text{-range} & v\text{ lower bound}\\ \hline
 $[0.002,0.1]$ & 5\cdot10^{-5}\\
 $[0.1,0.34]$ & 5\cdot10^{-5}\\
 $[0.34,0.40]$ & 5\cdot10^{-5}\\
 $[0.40,1.4]$ & 5\cdot10^{-5}\\
 $[1.4,4.1]$ & 0.1
\end{array}
\]
Proposition~\ref{prop:two-construction-certificate} gives, for each row of the table, a finite cover by such boxes. Boxes mapping into $\mathcal Q_0$ are covered by the strict-convexity argument above. On every remaining box the right-hand side of \eqref{eq:mean-value} has a certified positive lower bound; the smallest among the five regional minima in Proposition~\ref{prop:two-construction-certificate} is
\[
 1.2396204165999228\cdot10^{-6}>0.
\]
Hence $\widetilde F>0$ on every such box. Together with the three analytic regions of Lemmas~\ref{lem:F-small-t} and~\ref{lem:F-boundary-tail}, these covers exhaust the domain $t>0$, $0\le d\le t$. Thus $F(t,d)\ge0$ everywhere.
\end{proof}

\begin{proof}[Proof of Lemma~\ref{lem:S}]
By Proposition~\ref{prop:F-reduction} and Lemma~\ref{lem:F-global}, inequality~\eqref{eq:S} holds throughout the open square.  Both sides of \eqref{eq:S} are continuous on $[0,1]^2$, so the boundary inequality follows by approximation from points in $(0,1)^2$.
\end{proof}

\section{Proof of Theorem~\ref{thm:main}}

This section proves Theorem~\ref{thm:main} using the independent protocol and Liu's Example~5 at
\begin{equation}
 \ell_0:=\frac{1409}{2000},
 \qquad
 \tau_0:=\ell_0^2=\frac{1985281}{4000000}.
 \label{eq:tuned-ell}
\end{equation}
Set
\begin{equation}
 m_0:=\frac{3085573725166011}{5000000000000000}=0.6171147450332022
 \label{eq:tuned-m}
\end{equation}
and use the weights
\begin{equation}
 \alpha_0:=\frac1{2m_0}+\frac1{800000000},
 \qquad
 \beta_0:=1-\alpha_0.
 \label{eq:tuned-weights}
\end{equation}
Since $m_0>1/2$, one has $1/(2m_0)<1$.  Moreover the exact rational inequality
\[
 2m_0\left(1-\frac1{800000000}\right)>1
\]
gives $1/(2m_0)<1-1/800000000$, and hence $0<\alpha_0<1$ and $0<\beta_0<1$.

Lemma~\ref{lem:tuned-product} proves
\[
 H(r_{\ell_0}(x,y))\ge
 \sqrt{H(g_{\ell_0}(x))H(g_{\ell_0}(y))},
\]
which is the pointwise estimate used for the Example~5 term.  Lemma~\ref{lem:tuned-second} proves the companion inequality involving the independent term and the same one-variable factors.  Proposition~\ref{prop:tuned-fixed} then integrates these two inequalities.

\begin{lemma}
\label{lem:tuned-product}
For every $x,y\in[0,1]$,
\begin{equation}
 H(r_{\ell_0}(x,y))^2
 \ge H(g_{\ell_0}(x))H(g_{\ell_0}(y)).
 \label{eq:tuned-product}
\end{equation}
\end{lemma}

\begin{proof}
Write
\[
 W:=1+\tau_0(1-x)(1-y),
 \qquad r:=r_{\ell_0}(x,y)=xyW.
\]
Direct differentiation gives
\[
 r_x=y\bigl[1+\tau_0(1-2x)(1-y)\bigr],
 \qquad
 r_y=x\bigl[1+\tau_0(1-x)(1-2y)\bigr],
\]
and
\[
 r_{xy}=1+\tau_0(1-2x)(1-2y).
\]
The identity
\begin{align*}
 &W\bigl[1+\tau_0(1-2x)(1-2y)\bigr]\\
 &\quad-\bigl[1+\tau_0(1-2x)(1-y)\bigr]
       \bigl[1+\tau_0(1-x)(1-2y)\bigr]
 =\tau_0xy
\end{align*}
therefore yields, for $0<x,y<1$,
\begin{equation}
 W r_{xy}=\frac{r_xr_y}{xy}+\tau_0xy.
 \label{eq:tuned-r-identity}
\end{equation}
Because $\tau_0<1$,
\[
 W>0,
 \qquad
 1+\tau_0(1-2x)(1-y)\ge1-\tau_0>0,
\]
and the analogous inequality holds for the factor in $r_y$.  Thus $r,r_x,r_y>0$ in the open square, and division in \eqref{eq:tuned-r-identity} gives
\begin{equation}
 \frac{r_{xy}}{r_xr_y}
 =\frac1r+\frac{\tau_0xy}{Wr_xr_y}.
 \label{eq:tuned-r-ratio}
\end{equation}

The function $x\mapsto g_{\ell_0}(x)=r_{\ell_0}(x,x)$ is strictly increasing on $(0,1)$, because
\[
 g_{\ell_0}'(x)
 =2x\left[1+\tau_0(1-x)(1-2x)\right]
 \ge2x\left(1-\frac{\tau_0}{8}\right)>0,
\]
where $(1-x)(1-2x)\ge-1/8$ on $[0,1]$.

By symmetry assume $0<x<y<1$, and put
\[
 L(x,y):=\log H(r_{\ell_0}(x,y)),
 \qquad
 \Delta_x(y):=L(x,y)-\frac12L(x,x)-\frac12L(y,y).
\]
Since $\Delta_x(x)=0$, it is enough to prove $\partial_y\Delta_x(y)>0$.  Differentiating the last term along the diagonal gives
\[
 \frac{\dd}{\dd y}L(y,y)=L_1(y,y)+L_2(y,y)=2L_2(y,y),
\]
where the last equality follows from the symmetry of $L$.  Hence
\begin{equation}
 \partial_y\Delta_x(y)=L_2(x,y)-L_2(y,y).
 \label{eq:tuned-Delta-derivative}
\end{equation}
We prove that the right-hand side is positive.

Assume first that $g_{\ell_0}(y)<1/2$.  Since $g_{\ell_0}$ is increasing and
\[
 g_{\ell_0}(45/64)-\frac12
 =\frac{42978321721}{2684354560000}>0,
\]
we have $x,y<45/64$.  Proposition~\ref{prop:tuned-bernstein} gives, on this square,
\[
 Wr_xr_y\ge\tau_0xy(1-r).
\]
Substituting this inequality in \eqref{eq:tuned-r-ratio} gives
\begin{equation}
 \frac{r_{xy}}{r_xr_y}
 \le\frac1r+\frac1{1-r}.
 \label{eq:tuned-r-ratio-bound}
\end{equation}
For every $s\in[x,y]$,
\[
 r_{\ell_0}(s,y)\le r_{\ell_0}(y,y)=g_{\ell_0}(y)<\frac12.
\]
With $\Lambda(u)=\log H(u)$ as in Section~4, the chain rule gives
\[
 L_{12}
 =\Lambda'(r)r_xr_y
 \left(
 \frac{r_{xy}}{r_xr_y}+\frac{\Lambda''(r)}{\Lambda'(r)}
 \right).
\]
Here $\Lambda'(r)r_xr_y>0$.  Equation~\eqref{eq:tuned-r-ratio-bound} and Lemma~\ref{lem:entropy-one-variable} show that the expression in parentheses is strictly negative.  Therefore
\[
 L_{12}(s,y)<0\qquad(x\le s\le y).
\]
Integrating with respect to the first variable,
\[
 L_2(y,y)-L_2(x,y)=\int_x^yL_{12}(s,y)\,\dd s<0.
\]
Thus $L_2(x,y)>L_2(y,y)$, which proves \eqref{eq:tuned-Delta-derivative} in this case.

Assume now that $g_{\ell_0}(y)\ge1/2$.  Again the goal is to show $L_2(x,y)>L_2(y,y)$.  Since
\[
 \frac{r_y}{r}
 =\frac1y-
 \frac{\tau_0(1-x)}{1+\tau_0(1-x)(1-y)},
\]
write this positive coefficient as
\[
 A_x(y):=\frac{r_y(x,y)}{r(x,y)}.
\]
The function
\[
 a\longmapsto\frac{\tau_0a}{1+\tau_0a(1-y)}
\]
is strictly increasing for $a\ge0$.  Since $1-x>1-y$, subtracting its values from $1/y$ gives
\[
 0<A_x(y)<A_y(y).
\]
Also $r_x>0$, so
\[
 r(x,y)<r(y,y)=g_{\ell_0}(y).
\]
Finally, by the definition of $M$ in Section~4,
\[
 L_2(x,y)=A_x(y)M(r(x,y)),
 \qquad
 L_2(y,y)=A_y(y)M(g_{\ell_0}(y)).
\]
There are two possibilities.  If $r(x,y)<1/2\le g_{\ell_0}(y)$, then
\[
 M(r(x,y))>0\ge M(g_{\ell_0}(y)),
\]
so $L_2(x,y)>L_2(y,y)$.  If instead
\[
 \frac12\le r(x,y)<g_{\ell_0}(y),
\]
Lemma~\ref{lem:M-decreasing} gives
\[
 M(g_{\ell_0}(y))<M(r(x,y))\le0.
\]
Multiplying the first strict inequality by the positive number $A_x(y)$ and then using $A_x(y)<A_y(y)$ together with $M(g_{\ell_0}(y))<0$ gives
\[
 A_x(y)M(r(x,y))
 >A_x(y)M(g_{\ell_0}(y))
 >A_y(y)M(g_{\ell_0}(y)).
\]
Thus $L_2(x,y)>L_2(y,y)$ in this subcase as well.

Equation~\eqref{eq:tuned-Delta-derivative} is therefore positive for every $0<x<y<1$, and hence $\Delta_x(y)>0$.  By symmetry,
\[
 2\log H(r_{\ell_0}(x,y))
 \ge\log H(g_{\ell_0}(x))+\log H(g_{\ell_0}(y))
\]
throughout the open square, with equality when $x=y$.  Exponentiating and using continuity gives \eqref{eq:tuned-product} on $[0,1]^2$.
\end{proof}

The next lemma is proved on three analytic parts of the domain; Appendix~A handles the compact set that remains.

\begin{lemma}
\label{lem:tuned-second}
For every $x,y\in[0,1]$,
\begin{equation}
 2\alpha_0H(xy)
 +2\beta_0\sqrt{H(g_{\ell_0}(x))H(g_{\ell_0}(y))}
 \ge\frac{xH(y)+yH(x)}{m_0}.
 \label{eq:tuned-second}
\end{equation}
\end{lemma}

\begin{proof}
Both sides of \eqref{eq:tuned-second} are symmetric in $x,y$ and continuous on $[0,1]^2$.  It is therefore enough to prove the strict inequality for $0<y\le x<1$; the case $y>x$ follows by interchanging the variables, and the boundary follows by continuity.  Write
\[
 x=e^{-s},\qquad y=e^{-\sigma}.
\]
Since $u\mapsto e^{-u}$ is decreasing, $0<s\le\sigma$.

Using $G(u)=e^uH(e^{-u})$ from Section~4, define only the additional one-variable function
\begin{equation}
 W_0(u):=e^{2u}H(g_{\ell_0}(e^{-u}))
 \qquad(u>0).
 \label{eq:W0-def}
\end{equation}
The three terms in \eqref{eq:tuned-second}, after division by $2xy$, become
\[
 \frac{H(xy)}{xy}=G(s+\sigma),
\]
\[
 \frac{\sqrt{H(g_{\ell_0}(x))H(g_{\ell_0}(y))}}{xy}
 =\sqrt{W_0(s)W_0(\sigma)},
\]
and
\[
 \frac{xH(y)+yH(x)}{2m_0xy}
 =\frac{G(s)+G(\sigma)}{2m_0}.
\]
Thus \eqref{eq:tuned-second} is equivalent to
\begin{equation}
 F_0(s,\sigma):=
 \alpha_0G(s+\sigma)
 -\frac{G(s)+G(\sigma)}{2m_0}
 +\beta_0\sqrt{W_0(s)W_0(\sigma)}>0.
 \label{eq:tuned-F}
\end{equation}

We shall use that both $G$ and $W_0$ are increasing.  The series from Section~5 gives
\begin{align}
 G(u)&=u+1-\sum_{k\ge1}\frac{e^{-ku}}{k(k+1)},\notag\\
 G'(u)&=1+\sum_{k\ge1}\frac{e^{-ku}}{k+1}>0,\notag\\
 G''(u)&=-\sum_{k\ge1}\frac{k e^{-ku}}{k+1}<0.
 \label{eq:G-series-tuned}
\end{align}
Hence $G$ is strictly increasing and strictly concave.

To prove that $W_0$ is increasing, fix $v>0$, put $z=e^{-v}$, and set
\[
 q:=g_{\ell_0}(z)=z^2\bigl(1+\tau_0(1-z)^2\bigr).
\]
Differentiating \eqref{eq:W0-def} and using $\dd z/\dd v=-z$ gives
\[
 W_0'(v)=\frac{2H(q)-zq_zH'(q)}{z^2},
\]
where
\[
 zq_z=2q-2\tau_0z^3(1-z).
\]
Since $H(q)-qH'(q)=-\log(1-q)$,
\[
 2H(q)-zq_zH'(q)
 =-2\log(1-q)+2\tau_0z^3(1-z)H'(q).
\]
If $q\le1/2$, then $H'(q)\ge0$, and the right-hand side is positive.  If $q>1/2$, then
\[
 -H'(q)=\log\frac q{1-q}\le-\log(1-q)
\]
and
\[
 \tau_0z^3(1-z)\le\frac{27\tau_0}{256}<1.
\]
Because $\log(1-q)<0$,
\[
 2\tau_0z^3(1-z)H'(q)
 \ge2\tau_0z^3(1-z)\log(1-q)
 >2\log(1-q),
\]
so the right-hand side is again positive.  Therefore $W_0'(v)>0$ for every $v>0$.

For the remaining estimates put
\[
 D:=G(s)+G(\sigma)-G(s+\sigma).
\]
Using \eqref{eq:tuned-weights}, equation~\eqref{eq:tuned-F} becomes
\begin{equation}
 F_0(s,\sigma)
 =\frac{G(s+\sigma)}{800000000}
 -\frac{D}{2m_0}
 +\beta_0\sqrt{W_0(s)W_0(\sigma)}.
 \label{eq:tuned-F-D}
\end{equation}
Concavity gives $D\ge0$: with $\lambda=s/(s+\sigma)$ and $G(0)=0$,
\[
 G(s)\ge\lambda G(s+\sigma),
 \qquad
 G(\sigma)\ge(1-\lambda)G(s+\sigma).
\]
The series \eqref{eq:G-series-tuned} gives
\[
 D=1-\sum_{k\ge1}
 \frac{e^{-ks}+e^{-k\sigma}-e^{-k(s+\sigma)}}{k(k+1)}.
\]
For every $k\ge1$ the numerator in the summand is at least both $e^{-ks}$ and $e^{-k\sigma}$.  Comparing term by term with the series for $G(s)-s$ and $G(\sigma)-\sigma$ yields
\begin{equation}
 0\le D\le\min\{G(s)-s,G(\sigma)-\sigma\}<1.
 \label{eq:tuned-D-basic}
\end{equation}

Proposition~\ref{prop:tuned-one-dimensional} gives, for every $v>0$,
\begin{equation}
 W_0(v)\ge\frac85G(v),
 \qquad
 G(16v)\ge\frac{178}{25}G(v),
 \label{eq:tuned-global-1d}
\end{equation}
and the three numerical inequalities used below.

\medskip
\noindent\textbf{Case 1.} Suppose
\[
 G(\sigma)\ge\frac{178}{25}G(s).
\]
From \eqref{eq:tuned-D-basic}, $D<G(s)$, while \eqref{eq:tuned-global-1d} gives
\[
 \sqrt{W_0(s)W_0(\sigma)}
 \ge\frac85\sqrt{G(s)G(\sigma)}.
\]
Substitution in \eqref{eq:tuned-F-D} gives
\begin{align*}
 F_0(s,\sigma)
 &\ge\frac{G(s+\sigma)}{800000000}
 -\frac{G(s)}{2m_0}
 +\frac{8\beta_0}{5}\sqrt{G(s)G(\sigma)}\\
 &\ge\frac{G(s+\sigma)}{800000000}
 +G(s)\left(
 \frac{8\beta_0}{5}\sqrt{\frac{178}{25}}-\frac1{2m_0}
 \right)>0,
\end{align*}
because the last parenthesis is positive by Proposition~\ref{prop:tuned-one-dimensional}.

\medskip
\noindent\textbf{Case 2.} Suppose $s,\sigma\ge2$.  Then $D<1$ and, since $W_0$ is increasing,
\[
 \sqrt{W_0(s)W_0(\sigma)}\ge W_0(2).
\]
Equation~\eqref{eq:tuned-F-D} therefore gives
\[
 F_0(s,\sigma)
 >\frac{G(s+\sigma)}{800000000}
 -\frac1{2m_0}+\beta_0W_0(2)>0,
\]
where the last inequality again comes from Proposition~\ref{prop:tuned-one-dimensional}.

\medskip
\noindent\textbf{Case 3.} Suppose $s,\sigma\le2^{-10}$.  Put $u:=s+\sigma$.  With the same
\[
 \lambda=\frac{s}{s+\sigma}
\]
used above, we have $s=\lambda u$, $\sigma=(1-\lambda)u$, and $u\le2^{-9}$.  The function
\[
 (a,b)\longmapsto G(a)+G(b)-G(a+b)
\]
vanishes when $a=0$ or $b=0$, because $G(0)=0$, and its mixed derivative is $-G''(a+b)$.  Applying the fundamental theorem of calculus first in $b$ and then in $a$, the boundary terms on the two coordinate axes vanish and we obtain
\[
 D=-\int_0^s\int_0^\sigma G''(a+b)\,\dd b\,\dd a.
\]
By \eqref{eq:G-series-tuned},
\[
 -G''(v)
 <\sum_{k\ge1}e^{-kv}
 =\frac1{e^v-1}<\frac1v.
\]
Therefore
\[
 D<\int_0^s\int_0^\sigma\frac{\dd b\,\dd a}{a+b}.
\]
The integrand is nonnegative, so Tonelli's theorem allows us to integrate first in $b$.  For every $a>0$,
\[
 \int_0^\sigma\frac{\dd b}{a+b}
 =\log(a+\sigma)-\log a.
\]
The function on the right is integrable on $(0,s]$: the first logarithm is bounded there, and $\int_0^s|\log a|\,\dd a<\infty$.  Hence
\[
 \int_0^s\int_0^\sigma\frac{\dd b\,\dd a}{a+b}
 =\int_0^s\bigl(\log(a+\sigma)-\log a\bigr)\,\dd a.
\]
For $0<\varepsilon<s$, direct integration gives
\begin{align*}
 &\int_\varepsilon^s\bigl(\log(a+\sigma)-\log a\bigr)\,\dd a\\
 &\quad=\Bigl[(a+\sigma)\log(a+\sigma)-(a+\sigma)-a\log a+a\Bigr]_{a=\varepsilon}^{a=s}\\
 &\quad=(s+\sigma)\log(s+\sigma)-(\varepsilon+\sigma)\log(\varepsilon+\sigma)
       -s\log s+\varepsilon\log\varepsilon.
\end{align*}
Moreover,
\[
 \lim_{\varepsilon\to0^+}\varepsilon\log\varepsilon=0,
 \qquad
 \lim_{\varepsilon\to0^+}(\varepsilon+\sigma)\log(\varepsilon+\sigma)
 =\sigma\log\sigma.
\]
Taking the limit $\varepsilon\to0^+$ therefore gives
\begin{align*}
 \int_0^s\int_0^\sigma\frac{\dd b\,\dd a}{a+b}
 &=(s+\sigma)\log(s+\sigma)-s\log s-\sigma\log\sigma\\
 &=uH(\lambda).
\end{align*}
Thus
\begin{equation}
 D<uH(\lambda).
 \label{eq:tuned-D-entropy}
\end{equation}

For $w\in\{s,\sigma\}$ set
\[
 r_w:=1-g_{\ell_0}(e^{-w}).
\]
Since $\tau_0<1$, $1-e^{-w}\le w$, and $e^{2w}\ge1+2w+2w^2$,
\begin{align*}
 e^{2w}r_w
 &=e^{2w}-1-\tau_0(1-e^{-w})^2\\
 &\ge 2w+2w^2-w^2\ge2w.
\end{align*}
Also $g_{\ell_0}(e^{-w})\ge e^{-2w}$, so
\[
 r_w\le1-e^{-2w}\le2w\le2u.
\]
Using $H(1-r_w)=H(r_w)$ in \eqref{eq:W0-def},
\[
 W_0(w)=e^{2w}H(r_w)
 \ge e^{2w}r_w\log\frac1{r_w}
 \ge2w\log\frac1{2u}.
\]
Consequently
\begin{equation}
 \beta_0\sqrt{W_0(s)W_0(\sigma)}
 \ge2\beta_0u\sqrt{\lambda(1-\lambda)}\log\frac1{2u}.
 \label{eq:tuned-W-small}
\end{equation}
We shall also use
\begin{equation}
 H(\lambda)\le2\sqrt{\lambda(1-\lambda)}
 \qquad(0\le\lambda\le1).
 \label{eq:H-sqrt}
\end{equation}
For $0<\lambda<1$, the inequality $\log t\le t-1$ gives
\[
 -\lambda\log\lambda
 =2\lambda\log\frac1{\sqrt\lambda}
 \le2(\sqrt\lambda-\lambda),
\]
and similarly
\[
 -(1-\lambda)\log(1-\lambda)
 \le2(\sqrt{1-\lambda}-(1-\lambda)).
\]
Adding and using
\[
 \sqrt\lambda+\sqrt{1-\lambda}-1\le\sqrt{\lambda(1-\lambda)}
\]
proves \eqref{eq:H-sqrt}; the endpoints follow by continuity.

Since $u\le2^{-9}$,
\[
 \log\frac1{2u}\ge8\log2,
\]
and Proposition~\ref{prop:tuned-one-dimensional} gives
\[
 8\beta_0\log2>\frac1{2m_0}.
\]
Combining this inequality with \eqref{eq:tuned-D-entropy}, \eqref{eq:tuned-W-small}, and \eqref{eq:H-sqrt} in \eqref{eq:tuned-F-D} yields
\begin{align*}
 F_0(s,\sigma)
 &>\frac{G(u)}{800000000}
 -\frac{uH(\lambda)}{2m_0}
 +2\beta_0u\sqrt{\lambda(1-\lambda)}\log\frac1{2u}\\
 &\ge\frac{G(u)}{800000000}
 +2u\sqrt{\lambda(1-\lambda)}
 \left(
 \beta_0\log\frac1{2u}-\frac1{2m_0}
 \right)>0.
\end{align*}

It remains to consider points satisfying none of Cases~1--3.  Since Case~1 does not apply,
\[
 G(\sigma)<\frac{178}{25}G(s).
\]
If $\sigma\ge16s$, monotonicity of $G$ and \eqref{eq:tuned-global-1d} would give
\[
 G(\sigma)\ge G(16s)\ge\frac{178}{25}G(s),
\]
a contradiction.  Hence $\sigma<16s$.  Because Case~2 does not apply and $s\le\sigma$, we cannot have $s\ge2$; otherwise $\sigma\ge2$ as well.  Thus $s<2$.  Because Case~3 does not apply, $s$ and $\sigma$ are not both at most $2^{-10}$; since $s\le\sigma$, this implies $\sigma>2^{-10}$.  Therefore
\[
 2^{-14}<s<2,
 \qquad
 2^{-10}<\sigma<32.
\]
Set
\[
 t:=\frac{s+\sigma}{2},
 \qquad
 \xi:=\frac{\sigma-s}{\sigma+s},
 \qquad
 d:=t\xi.
\]
The preceding bounds give $t>\sigma/2>2^{-11}$ and $t<(2+32)/2=17$.  Moreover $\sigma/s<16$, so
\[
 0\le\xi=\frac{\sigma/s-1}{\sigma/s+1}<\frac{15}{17}.
\]
Thus
\[
 2^{-11}<t<17,
 \qquad
 0\le\xi<\frac{15}{17},
\]
and $s=t-d=t(1-\xi)$, $\sigma=t+d=t(1+\xi)$.

Proposition~\ref{prop:tuned-contact} proves
\[
 F_0(t-d,t+d)>0
\]
whenever
\[
 |t-0.369781418829901|\le\frac1{20},
 \qquad |d|\le\frac1{20}.
\]
Proposition~\ref{prop:tuned-compact} gives a finite cover of
\[
 [2^{-11},17]\times[0,15/17]
\]
such that every cell either maps into this rectangle under $d=t\xi$, or has a strictly positive outward-rounded mean-value lower bound for $F_0(t(1-\xi),t(1+\xi))$.  Hence $F_0>0$ at every remaining point.  This proves the strict inequality in the open square.  Continuity gives the non-strict inequality on the boundary.
\end{proof}

\begin{proposition}
\label{prop:tuned-fixed}
Let $P$ be a probability measure on $[0,1]$ and let $\rho\in\cC_3(P)$.  Then
\[
 \alpha_0\iint H(xy)\,P(\dd x)P(\dd y)
 +\beta_0\iint H(r_{\ell_0}(x,y))\,\rho(\dd x,\dd y)
 \ge\frac{m(P)}{m_0}\int H\,\dd P.
\]
\end{proposition}

\begin{proof}
Lemma~\ref{lem:tuned-product} and nonnegativity give
\[
 H(r_{\ell_0}(x,y))
 \ge\sqrt{H(g_{\ell_0}(x))}\sqrt{H(g_{\ell_0}(y))}.
\]
Integrating with respect to $\rho$ and applying Lemma~\ref{lem:positive-correlation} to the continuous function $x\mapsto\sqrt{H(g_{\ell_0}(x))}$ gives
\begin{equation}
 \iint H(r_{\ell_0}(x,y))\,\rho(\dd x,\dd y)
 \ge
 \left(\int\sqrt{H(g_{\ell_0}(x))}\,P(\dd x)\right)^2.
 \label{eq:tuned-integrated-product}
\end{equation}

Integrate Lemma~\ref{lem:tuned-second} with respect to $P(\dd x)P(\dd y)$ and divide by $2$.  Fubini gives
\[
 \iint xH(y)\,P(\dd x)P(\dd y)
 =m(P)\int H\,\dd P
 =\iint yH(x)\,P(\dd x)P(\dd y),
\]
while the product term factors as
\[
 \iint\sqrt{H(g_{\ell_0}(x))H(g_{\ell_0}(y))}\,
 P(\dd x)P(\dd y)
 =\left(\int\sqrt{H(g_{\ell_0}(x))}\,P(\dd x)\right)^2.
\]
Therefore
\begin{equation}
 \alpha_0\iint H(xy)\,P(\dd x)P(\dd y)
 +\beta_0\left(\int\sqrt{H(g_{\ell_0}(x))}\,P(\dd x)\right)^2
 \ge\frac{m(P)}{m_0}\int H\,\dd P.
 \label{eq:tuned-integrated-second}
\end{equation}
Replacing the square in \eqref{eq:tuned-integrated-second} by the larger integral from \eqref{eq:tuned-integrated-product} proves the proposition.
\end{proof}

\begin{proof}[Proof of Theorem~\ref{thm:main}]
Fix $0<c<1-m_0$ and let $\mu$ be a probability measure with $\int s\,\mu(\dd s)\le c$.  Let $S$ be a $[0,1]$-valued random variable with distribution $\mu$, and let $P$ be the distribution of $1-S$.  Then
\[
 m(P)=\int(1-s)\,\mu(\dd s)
 =1-\int s\,\mu(\dd s)
 \ge1-c>m_0.
\]
We apply Liu's criterion to the independent protocol and Example~5 with parameter $\ell_0$, using the weights $\alpha_0,\beta_0$.  For the independent protocol, if $S_1,S_2$ are independent with law $\mu$, then $1-S_1,1-S_2$ are independent with law $P$.  Thus the coupling $\mu\otimes\mu$ becomes $P\otimes P$, and its entropy integrand becomes $H(xy)$.  For Example~5, let $\rho\in\cC_3(\mu)$ and let $(X,Y)$ have joint law $\rho$.  The law of $(1-X,1-Y)$ has both marginals equal to $P$ and belongs to $\cC_3(P)$: in \eqref{eq:C3-def}, each product measure $Q^{\otimes2}$ is sent to the square of the one-dimensional law of $1-X$ when $X$ has distribution $Q$, and reflection preserves convex combinations and weak limits.  In these variables the entropy integrand is $H(r_{\ell_0}(x,y))$.

Proposition~\ref{prop:tuned-fixed} therefore bounds the weighted sum in Liu's criterion below by
\[
 \frac{m(P)}{m_0}\int H\,\dd P.
\]
Since $H(1-s)=H(s)$,
\[
 \int H\,\dd P=\int H\,\dd\mu,
\]
and
\[
 \frac{m(P)}{m_0}\ge\frac{1-c}{m_0}>1.
\]
Thus Liu's criterion applies with the {constant} factor $(1-c)/m_0$, and gives $c\le c_{\mathrm{UC}}$.  This holds for every $0<c<1-m_0$, so
\[
 c_{\mathrm{UC}}\ge1-m_0=0.3828852549667978.
\]
\end{proof}

\section{Limits of product estimates}

{This section records two limitations for estimates obtained by replacing each entropy term by a product $\psi(x)\psi(y)$.  We first treat the inequality used in the proof of Theorem~\ref{thm:liu-bound}.  A broader but weaker restriction is given afterwards.}

Let $\Pi$ be a conditionally i.i.d.\ protocol, and let $\psi:[0,1]\to\mathbb R$ be continuous.  We will use the pointwise estimate
\begin{equation}
 H\bigl(\Pi_{1-x,1-y}(0,0)\bigr)
 \ge\psi(x)\psi(y)
 \qquad(0\le x,y\le1).
 \label{eq:product-lower-bound}
\end{equation}
Whenever \eqref{eq:product-lower-bound} holds, Lemma~\ref{lem:positive-correlation} gives, for every probability measure $P$ and every $\rho\in\cC_3(P)$,
\[
 \iint H\bigl(\Pi_{1-x,1-y}(0,0)\bigr)\,\rho(\dd x,\dd y)
 \ge\iint\psi(x)\psi(y)\,\rho(\dd x,\dd y)
 \ge\left(\int\psi\,\dd P\right)^2.
\]

\subsection{An upper bound for the inequality used in Theorem~\ref{thm:liu-bound}}

Whenever a function $\psi$ satisfies \eqref{eq:product-lower-bound}, taking $y=x$ gives $\psi(x)^2\le\log2$, while $x=y=1$ gives $\psi(1)=0$.  We therefore consider the following general form of the inequality used in Theorem~\ref{thm:liu-bound}.

Put
\[
 J(x):=xH(x)-\frac12H(x^2),
 \qquad
 \widehat m:=\frac12+\frac{\max_{0\le x\le1}J(x)}{\log2},
 \qquad
 \widehat c:=1-\widehat m.
\]

\begin{theorem}
\label{thm:ceiling}
Let $m>0$.  Suppose that, for some $N\ge1$, there are weights
\[
 w_0,w_1,\ldots,w_N\ge0,
 \qquad \sum_{j=0}^Nw_j=1,
\]
and continuous functions $\psi_j:[0,1]\to\mathbb R$, $1\le j\le N$, satisfying
\[
 \psi_j(1)=0,
 \qquad
 \psi_j(x)^2\le\log2\quad(0\le x\le1),
\]
such that every probability measure $P$ on $[0,1]$ satisfies
\begin{equation}
 w_0\iint H(xy)\,P(\dd x)P(\dd y)
 +\sum_{j=1}^Nw_j\left(\int\psi_j\,\dd P\right)^2
 \ge \frac{m(P)}m\int H\,\dd P.
 \label{eq:mw-certificate}
\end{equation}
Then $m\ge\widehat m$.  Consequently
\[
 1-m\le\widehat c=0.3828852599667\ldots.
\]
\end{theorem}

For Theorem~\ref{thm:liu-bound}, this corresponds to $N=1$, $m=m_*$, $w_0=\alpha$, $w_1=\beta$, and $\psi_1(x)=\sqrt{H(g(x))}$; Lemma~\ref{lem:S} gives \eqref{eq:mw-certificate}.

\begin{proof}[Proof of Theorem~\ref{thm:ceiling}]
Because \eqref{eq:mw-certificate} holds for every probability measure, fix $x\in[0,1]$ and take $P=\delta_x$, the Dirac probability measure assigning mass $1$ to $x$.  For this measure,
\[
 m(P)=x,\qquad
 \iint H(uv)\,P(\dd u)P(\dd v)=H(x^2),
\]
\[
 \int\psi_j\,\dd P=\psi_j(x),
 \qquad
 \int H\,\dd P=H(x).
\]
Substitution in \eqref{eq:mw-certificate}, followed by multiplication by $m>0$, gives
\begin{equation}
 m\left[
 w_0H(x^2)+\sum_{j=1}^Nw_j\psi_j(x)^2
 \right]\ge xH(x).
 \label{eq:point-mass-probe}
\end{equation}

We next obtain a lower bound on the independent weight $w_0$.  Fix $y\in(0,1)$ and, for $0<\eta<1$, set
\[
 P_\eta=(1-\eta)\delta_1+\eta\delta_y.
\]
Expanding the product measure $P_\eta\otimes P_\eta$ gives
\begin{align*}
 \iint H(uv)\,P_\eta(\dd u)P_\eta(\dd v)
 &=(1-\eta)^2H(1)
   +2\eta(1-\eta)H(y)
   +\eta^2H(y^2)\\
 &=2\eta(1-\eta)H(y)+\eta^2H(y^2).
\end{align*}
For each $j$,
\[
 \int\psi_j\,\dd P_\eta
 =(1-\eta)\psi_j(1)+\eta\psi_j(y)
 =\eta\psi_j(y),
\]
and also
\[
 m(P_\eta)=1-\eta(1-y),
 \qquad
 \int H\,\dd P_\eta=\eta H(y).
\]
Since $H(y)>0$, divide \eqref{eq:mw-certificate} by $\eta H(y)$.  The resulting inequality is
\begin{align*}
 &w_0\left(2(1-\eta)+\eta\frac{H(y^2)}{H(y)}\right)
 +\sum_{j=1}^Nw_j\eta\frac{\psi_j(y)^2}{H(y)}\\
 &\hspace{45mm}\ge
 \frac{1-\eta(1-y)}m.
\end{align*}
Letting $\eta$ tend to $0$ gives $2w_0\ge1/m$, hence
\begin{equation}
 w_0\ge\frac1{2m}.
 \label{eq:w0-floor}
\end{equation}

Return to \eqref{eq:point-mass-probe}.  Since $\psi_j(x)^2\le\log2$ and $\sum_{j=1}^Nw_j=1-w_0$,
\begin{align*}
 w_0H(x^2)+\sum_{j=1}^Nw_j\psi_j(x)^2
 &\le w_0H(x^2)+(1-w_0)\log2\\
 &=\log2+w_0\bigl(H(x^2)-\log2\bigr).
\end{align*}
The coefficient $H(x^2)-\log2$ is nonpositive.  Using \eqref{eq:w0-floor} therefore gives
\[
 w_0H(x^2)+\sum_{j=1}^Nw_j\psi_j(x)^2
 \le \log2+\frac{H(x^2)-\log2}{2m}.
\]
Combining this estimate with \eqref{eq:point-mass-probe} yields
\begin{align*}
 xH(x)
 &\le m\log2+\frac12\bigl(H(x^2)-\log2\bigr)\\
 &=\frac12H(x^2)+\left(m-\frac12\right)\log2.
\end{align*}
Thus
\[
 m\ge\frac12+\frac{xH(x)-\frac12H(x^2)}{\log2}
 =\frac12+\frac{J(x)}{\log2}
 \qquad(0\le x\le1).
\]
Taking the maximum over $x$ proves $m\ge\widehat m$, and hence $1-m\le\widehat c$.

It remains to justify the numerical value of $\widehat c$ stated in the theorem.  For $x\in(0,1)$,
\begin{align*}
 J'(x)&=-(1-x)\log(1-x)-x\log(1+x),\\
 J''(x)&=\log\frac{1-x}{1+x}+\frac1{1+x},\\
 J'''(x)&=\frac{x+3}{(x+1)^2(x-1)}<0.
\end{align*}
Thus $J''$ is strictly decreasing.  Moreover
\[
 \lim_{x\to0^+}J''(x)=1,
 \qquad
 \lim_{x\to1^-}J''(x)=-\infty,
\]
so $J''$ has a unique zero and $J'$ first increases and then decreases.  Since
\[
 \lim_{x\to0^+}J'(x)=0,
 \qquad
 \lim_{x\to1^-}J'(x)=-\log2<0,
\]
and $J''(x)>0$ for all sufficiently small $x>0$, we have $J'(x)>0$ for all sufficiently small $x>0$.  Therefore $J'$ has exactly one zero in $(0,1)$, and $J$ has a unique maximizer $\widehat x\in(0,1)$.  Proposition~\ref{prop:ceiling-certificate} gives rigorous intervals for $\widehat x$, $\widehat m$, and $\widehat c$, in particular
\[
 \widehat c=0.3828852599667\ldots.
\]
\end{proof}

For comparison with the bound in Theorem~\ref{thm:main}, Proposition~\ref{prop:ceiling-certificate} gives
\[
 0<\widehat c-(1-m_0)<5.1\cdot10^{-9}.
\]

\subsection{A broader product inequality}

Setting $y=x$ in \eqref{eq:product-lower-bound} gives
\begin{equation}
 \psi(x)^2
 \le H\bigl(\Pi_{1-x,1-x}(0,0)\bigr)
 \le\log2.
 \label{eq:general-diagonal-bound}
\end{equation}
At $x=1$, both Bernoulli marginals before complementation are $\operatorname{Ber}(0)$, the point mass at $0$.  There is therefore only one coupling of these marginals, namely the point mass at $(0,0)$; in particular $\Pi_{0,0}(0,0)=1$.  Substituting $x=1$ in \eqref{eq:general-diagonal-bound} gives
\begin{equation}
 \psi(1)^2\le H(1)=0,
 \qquad\text{hence}\qquad
 \psi(1)=0.
 \label{eq:general-psi-one}
\end{equation}

The function
\[
 x\longmapsto H(x)-\frac12H(x^2)-\frac12\log2
\]
is strictly decreasing on $(1/2,1)$, since its derivative is
\begin{align*}
 H'(x)-xH'(x^2)
 &=(1-x)\log\frac{1-x}{x}
   -x\log\frac{1+x}{x}<0.
\end{align*}
Its value at $x=1/2$ is positive, while its value at $x=1$ is $-\frac12\log2<0$.  Hence there is a unique $x_c\in(1/2,1)$ such that
\begin{equation}
 H(x_c)-\frac12H(x_c^2)=\frac12\log2.
 \label{eq:xc-def}
\end{equation}

{We finish with a restriction for a broader form of the product inequality, in which the right-hand side is multiplied by a constant $\gamma>1$ independent of $P$.}

\begin{proposition}
\label{prop:general-product-ceiling}
Let $m>0$ and $\gamma>1$.  Let $\Pi^{(1)},\ldots,\Pi^{(N)}$ be conditionally i.i.d.\ protocols, and let $\psi_j:[0,1]\to\mathbb R$ be continuous functions such that, for every $j=1,\ldots,N$,
\[
 H\bigl(\Pi^{(j)}_{1-x,1-y}(0,0)\bigr)
 \ge\psi_j(x)\psi_j(y)
 \qquad(0\le x,y\le1).
\]
Let
\[
 w_0,w_1,\ldots,w_N\ge0,
 \qquad \sum_{j=0}^Nw_j=1.
\]
Suppose that every probability measure $P$ on $[0,1]$ with $m(P)\ge m$ satisfies
\begin{equation}
 w_0\iint H(xy)\,P(\dd x)P(\dd y)
 +\sum_{j=1}^Nw_j\left(\int\psi_j\,\dd P\right)^2
 \ge\gamma\int H\,\dd P.
 \label{eq:general-product-certificate}
\end{equation}
Then $m>x_c$.  Consequently every threshold $c=1-m$ obtained from these hypotheses satisfies
\[
 c<1-x_c=0.40036117239\ldots.
\]
\end{proposition}

\begin{proof}
Since $x_c<1$, the conclusion $m>x_c$ already holds when $m\ge1$.  We may therefore assume $0<m<1$.

Fix $y\in(0,1)$.  For every
\[
 0<\eta<\min\left\{1,\frac{1-m}{1-y}\right\},
\]
define
\[
 P_\eta=(1-\eta)\delta_1+\eta\delta_y,
\]
where $\delta_z$ denotes the Dirac probability measure assigning mass $1$ to $z$.  Then
\[
 m(P_\eta)=1-\eta(1-y)>m,
\]
so $P_\eta$ is among the measures to which \eqref{eq:general-product-certificate} applies.  Because $H(1)=0$, the independent term is exactly
\[
 \iint H(xy)\,P_\eta(\dd x)P_\eta(\dd y)
 =2\eta(1-\eta)H(y)+\eta^2H(y^2).
\]
For each $j$, the pointwise assumption in the proposition is \eqref{eq:product-lower-bound} with $(\Pi,\psi)=(\Pi^{(j)},\psi_j)$; hence \eqref{eq:general-psi-one} gives
\[
 \int\psi_j\,\dd P_\eta=\eta\psi_j(y).
\]
Finally,
\[
 \int H\,\dd P_\eta=\eta H(y).
\]
Since $H(y)>0$, divide \eqref{eq:general-product-certificate} by $\eta H(y)$.  Substituting these expressions gives
\[
 w_0\left(2(1-\eta)+\eta\frac{H(y^2)}{H(y)}\right)
 +\eta\sum_{j=1}^N w_j\frac{\psi_j(y)^2}{H(y)}
 \ge\gamma.
\]
Because $y$ is fixed and $H(y)>0$, all ratios in this display are finite.  Taking the one-sided limit $\eta\to0^+$,
\[
 \lim_{\eta\to0^+}
 w_0\left(2(1-\eta)+\eta\frac{H(y^2)}{H(y)}\right)=2w_0,
\]
and
\[
 \lim_{\eta\to0^+}
 \eta\sum_{j=1}^N w_j\frac{\psi_j(y)^2}{H(y)}=0.
\]
The right-hand side is the constant $\gamma$, so the limiting inequality is $2w_0\ge\gamma$.  Hence
\[
 w_0\ge\frac\gamma2.
\]

Now let $x\ge m$ and apply \eqref{eq:general-product-certificate} to $P=\delta_x$.  For each $j$, the same specialization of \eqref{eq:general-diagonal-bound} gives $\psi_j(x)^2\le\log2$.  Since $\sum_{j=1}^Nw_j=1-w_0$, we obtain
\[
 \gamma H(x)
 \le w_0H(x^2)+(1-w_0)\log2.
\]
Because $H(x^2)\le\log2$, the coefficient of $w_0$ in the right-hand side is $H(x^2)-\log2\le0$.  Hence increasing $w_0$ can only decrease that right-hand side.  Since $w_0\ge\gamma/2$, its value is at most the value obtained at $w_0=\gamma/2$, and therefore
\[
 \gamma H(x)
 \le\frac\gamma2H(x^2)
   +\left(1-\frac\gamma2\right)\log2.
\]
Move $\frac\gamma2H(x^2)$ to the left, divide by $\gamma\log2>0$, and then add $1/2$ to both sides.  This gives
\begin{equation}
 \frac1\gamma
 \ge\frac12+
 \frac{H(x)-\frac12H(x^2)}{\log2}.
 \label{eq:general-point-probe}
\end{equation}
At $x=x_c$, equation~\eqref{eq:xc-def} makes the right-hand side of \eqref{eq:general-point-probe} equal to $1$.  The left-hand side is strictly smaller than $1$ because $\gamma>1$.  Hence \eqref{eq:general-point-probe} cannot hold at $x_c$.  If $m\le x_c$, then $P=\delta_{x_c}$ satisfies $m(P)=x_c\ge m$ and must satisfy the inequality, a contradiction.  Therefore $m>x_c$.

Proposition~\ref{prop:general-product-numerics} gives
\[
 1-x_c\in[0.40036117239,\ 0.40036117240],
\]
which is the stated numerical bound.
\end{proof}

\appendix
\section{Computational results}

The goal of this appendix is to state, in mathematical form, every finite computation used above.  Propositions~A.1, A.3, A.5--A.8 use outward-rounded interval arithmetic with MPFR evaluation of transcendental functions; Propositions~A.2 and~A.4 are exact rational Bernstein calculations.  The computational supplement describes the arithmetic model, the finite covers, and the proposition-specific ancillary files; both the supplement and the complete ancillary package are available in the Zenodo record associated with this work~\cite{joint-zenodo}.

\begin{certificate}
\label{prop:two-construction-certificate}
The constants in Section~3 satisfy
\begin{align*}
 x_*&\in[0.69078759392498,\ 0.69078759392500],\\
 m_*&\in[0.61729091208117,\ 0.61729091208136],\\
 \beta&\in[0.1000525598619,\ 0.1000525598639],\\
 t_*&\in[0.36992289186769,\ 0.36992289186773].
\end{align*}
The same certificate verifies
\[
 c_{\mathrm L}>0.382709087918735.
\]
With $c_0=m_*\alpha$ and
\[
 \zeta(t):=2c_0(1-\log(2t))-1-e^t(-\log t+t),
\]
the scalar bounds used in Lemmas~\ref{lem:F-small-t} and~\ref{lem:F-boundary-tail} are
\begin{align*}
 \zeta(0.002)&\ge0.016672474402352748,\\
 \left(c_0-\frac12\right)G(0.004)-\frac12G(5\cdot10^{-5})
 &\ge0.0011788376979667077,\\
 \left(c_0-\frac12\right)G(2.8)-\frac12G(0.1)
 &\ge0.035598095197518845,\\
 c_0(9.2-e^{-8.2})-5.1
 &\ge0.010717682423498683.
\end{align*}
On the rectangle \eqref{eq:local-rectangle},
\begin{align*}
 F_{tt}&\ge0.13459156779665421,\\
 F_{dd}&\ge0.87181314699157286,\\
 F_{tt}F_{dd}-F_{td}^2&\ge0.050114365218749776.
\end{align*}
For $a<b$ and $\nu>0$, put
\[
 \mathcal R(a,b;\nu):=
 \{(t,\xi):a\le t\le b,\ 0\le\xi\le1,\ t(1-\xi)\ge\nu\}.
\]
Each of
\[
 \mathcal R(0.002,0.1;5\cdot10^{-5}),
 \quad
 \mathcal R(0.1,0.34;5\cdot10^{-5}),
\]
\[
 \mathcal R(0.34,0.40;5\cdot10^{-5}),
 \quad
 \mathcal R(0.40,1.4;5\cdot10^{-5}),
 \quad
 \mathcal R(1.4,4.1;0.1)
\]
has a finite cover by the boxes used in \eqref{eq:mean-value}.  Every box either maps into the local rectangle $\mathcal Q_0$ or has a positive mean-value lower bound.  The smallest lower bounds in the five regions, in the listed order, are
\[
 1.4683343626197839\cdot10^{-4},\quad
 1.2396204165999228\cdot10^{-6},\quad
 5.6012230503014744\cdot10^{-4},
\]
\[
 6.9620924720995344\cdot10^{-6},\qquad
 4.4688168746738737\cdot10^{-3}.
\]
\end{certificate}

\begin{certificate}
\label{prop:bernstein-certificate}
After the rescaling $x=(43/64)s$, $y=(43/64)t$, the polynomial $WUV-(1-r)$ from Lemma~\ref{lem:WUV} has bidegree $(3,3)$.  In the product Bernstein basis $\{B_i^3(s)B_j^3(t):0\le i,j\le3\}$ its coefficient matrix is
\[
\begin{pmatrix}
7 & \frac{41}{12} & \frac{8221}{6144} & \frac{20653}{131072}\\[1mm]
\frac{41}{12} & \frac{18803}{9216} & \frac{402175}{393216} & \frac{1894129}{6291456}\\[1mm]
\frac{8221}{6144} & \frac{402175}{393216} & \frac{52130497}{75497472} & \frac{587913113}{1610612736}\\[1mm]
\frac{20653}{131072} & \frac{1894129}{6291456} & \frac{587913113}{1610612736} & \frac{6389233321}{17179869184}
\end{pmatrix}.
\]
Every coefficient is strictly positive.
\end{certificate}

\begin{certificate}
\label{prop:ceiling-certificate}
The unique maximizer $\widehat x$ of $J$ satisfies
\[
 \widehat x\in
 [0.69090773872540,\ 0.69090773872542].
\]
The corresponding constants satisfy
\[
 \widehat m\in
 [0.617114740033122,\ 0.617114740033283],
\]
\[
 \widehat c\in
 [0.382885259966717,\ 0.382885259966878].
\]
In particular,
\[
 0<
 \widehat c-0.3828852549667978
 <5.1\cdot10^{-9}.
\]
\end{certificate}

\begin{certificate}
\label{prop:tuned-bernstein}
Let $\tau_0=1985281/4000000$ and
\[
 W=1+\tau_0(1-x)(1-y).
\]
After the rescaling $x=(45/64)s$, $y=(45/64)t$, the polynomial
\[
 W\bigl[1+\tau_0(1-2x)(1-y)\bigr]
  \bigl[1+\tau_0(1-x)(1-2y)\bigr]
 -\tau_0(1-r_{\ell_0}(x,y))
\]
has bidegree $(3,3)$ and all $16$ product-Bernstein coefficients are strictly positive.  Their minimum is the exact rational number
\[
 \frac{4652778828628843621308587}
 {8388608000000000000000000}
 =0.5546544585977606\ldots.
\]
Since
\[
 (r_{\ell_0})_x
 =y\bigl[1+\tau_0(1-2x)(1-y)\bigr],
 \qquad
 (r_{\ell_0})_y
 =x\bigl[1+\tau_0(1-x)(1-2y)\bigr],
\]
this implies
\[
 W(r_{\ell_0})_x(r_{\ell_0})_y
 \ge \tau_0xy\bigl(1-r_{\ell_0}(x,y)\bigr)
\]
on $[0,45/64]^2$.
\end{certificate}

\begin{certificate}
\label{prop:tuned-one-dimensional}
With $G,W_0,m_0,\beta_0$ as in {Section~6}, one has, for every $v>0$,
\[
 W_0(v)\ge\frac85 G(v),
 \qquad
 G(16v)\ge\frac{178}{25}G(v).
\]
The three scalar margins used in the analytic cases of Lemma~\ref{lem:tuned-second} satisfy
\begin{align*}
 \frac{8\beta_0}{5}\sqrt{\frac{178}{25}}-\frac1{2m_0}
 &\ge2.8908190756\cdot10^{-6},\\
 \beta_0W_0(2)-\frac1{2m_0}
 &\ge0.4053615705664,\\
 8\beta_0\log2-\frac1{2m_0}
 &\ge0.2421300730523.
\end{align*}
For each of the first two global inequalities, the finite middle interval $[0.001,8]$ is covered by outward-rounded interval boxes, with respectively $233$ and $243$ terminal cells and no unresolved cell.  The complementary intervals $(0,0.001]$ and $[8,\infty)$ are covered by the analytic estimates on $(0,0.001]$ and $[8,\infty)$ documented in the computational supplement.
\end{certificate}

\begin{certificate}
\label{prop:tuned-contact}
Let
\[
 t_c=0.369781418829901,
 \qquad h=\frac1{20},
\]
and let
\[
 \mathcal B=[t_c-h,t_c+h]\times[-h,h].
\]
The Hessian of the function
\[
 (t,d)\longmapsto F_0(t-d,t+d)
\]
is positive definite throughout $\mathcal B$.  More precisely, an outward-rounded interval subdivision with $824$ terminal cells, $823$ splits, and maximum depth $10$ gives
\[
 \partial_{tt}F_0(t-d,t+d)\ge0.08943365573651689
\]
and a Hessian determinant at least
\[
 0.0012824495103952169
\]
on every terminal cell.  At $(t_c,0)$,
\[
 F_0(t_c,t_c)\ge1.1386421250314525\cdot10^{-8},
\]
and the absolute value of the derivative in the $t$ direction is at most
\[
 1.2656542480726788\cdot10^{-14}.
\]
The derivative in the $d$ direction is zero there by evenness.  The tangent inequality for a differentiable convex function therefore gives, throughout $\mathcal B$,
\[
 F_0(t-d,t+d)\ge1.1386420617487399\cdot10^{-8}>0.
\]
\end{certificate}

\begin{certificate}
\label{prop:tuned-compact}
Let
\[
 \mathcal Q=[2^{-11},17]\times[0,15/17].
\]
There is a finite axis-parallel subdivision of $\mathcal Q$ with the following property.  Every terminal cell is either contained in
\[
 \{(t,\xi):|t-t_c|\le h,\ 0\le t\xi\le h\},
\]
so that Proposition~\ref{prop:tuned-contact} applies with $d=t\xi$, or satisfies
\[
 F_0(t_0(1-\xi_0),t_0(1+\xi_0))
 -\sup_R|\partial_t\{F_0(t(1-\xi),t(1+\xi))\}|\,\Delta_t
 -\sup_R|\partial_\xi\{F_0(t(1-\xi),t(1+\xi))\}|\,\Delta_\xi
 >0
\]
for the cell
\[
 R=[t_0-\Delta_t,t_0+\Delta_t]
 \times[\xi_0-\Delta_\xi,\xi_0+\Delta_\xi].
\]
The supplied cover has
\[
 666\text{ mean-value terminal cells},\qquad
 15\text{ cells covered by }\mathcal B,\qquad
 553\text{ splits},
\]
with maximum depth $8$ and no unresolved cell.  The least positive mean-value lower bound is
\[
 5.2266673845026352\cdot10^{-8}.
\]
\end{certificate}

\begin{certificate}
\label{prop:general-product-numerics}
Let $x_c$ be the unique solution in $(1/2,1)$ of
\[
 H(x)-\frac12H(x^2)=\frac12\log2.
\]
Directed interval evaluation gives
\[
 x_c\in[0.59963882760,
        0.59963882761]
\]
and hence
\[
 1-x_c\in[0.40036117239,
          0.40036117240].
\]
In particular,
\[
 1-x_c<0.40036117241.
\]
\end{certificate}

\section*{Acknowledgments}
Preliminary versions of parts of this work appeared separately in~\cite{costa-preprint,sadhu-note}.

\section*{Declaration of generative AI and AI-assisted technologies in the writing process}
ChatGPT (OpenAI) was used as an aid for exploratory symbolic calculations, proof checking, drafting and checking computer code, and language editing.  The mathematical statements, proofs, numerical choices, and computer-assisted certificates remain the responsibility of the authors.

\end{document}